\documentclass{amsart}

\usepackage{amsmath,amssymb,amsfonts,amsthm}
\usepackage{mathrsfs}
\usepackage{graphicx}
\usepackage{multicol,multirow}
\usepackage{rotating}
\usepackage{appendix}
\usepackage{xcolor}
\usepackage[colorlinks=true,linkcolor=blue,citecolor=blue,urlcolor=blue]{hyperref}

\numberwithin{equation}{section}

\theoremstyle{plain}
\newtheorem{theorem}{Theorem}[section]
\newtheorem{lemma}[theorem]{Lemma}
\newtheorem{corollary}[theorem]{Corollary}
\newtheorem{proposition}[theorem]{Proposition}

\theoremstyle{definition}
\newtheorem{definition}[theorem]{Definition}

\theoremstyle{remark}

\def\bigno{\bigskip \noindent}
\def\medno{\medskip \noindent}
\def\smallno{\smallskip \noindent}
\def\bignobf#1{\bigskip \noindent \textbf{#1}}
\def\mednobf#1{\medskip \noindent \textbf{#1}}

\title[Two Regularity Problems on Analytic Tent Spaces]{Two Regularity Problems on Analytic Tent Spaces}
\thanks{X. Fang is supported by National Science and Technology Council (NSTC 114-2115-M-A49-003-MY3). S. Hou is supported by National Natural Science Foundation (NNSF) of China (No. 12371133). Q. Zhou is supported by National Natural Science Foundation (NNSF) of China (No. 12501162), China Postdoctoral Science Foundation (No. 2024M762280) and Natural Science Foundation of Jiangsu Province (No. BK20250832). X. Zhu is supported by National Natural Science Foundation (NNSF Tianyuan) of China (No. 12526614).}

\author[X. Fang]{Xiang Fang}
\address{National Yang Ming Chiao Tung University, Hsinchu, Taiwan (R.O.C.)}
\email{xfang@nycu.edu.tw}

\author[F. Guo]{Feng Guo}
\address{Nanjing University of Aeronautics and Astronautics, Nanjing, P. R. China}
\email{70207994@nuaa.edu.cn}

\author[S. Hou]{Shengzhao Hou}
\address{Soochow University, Suzhou, P. R. China}
\email{shou@suda.edu.cn}

\author[Q. Zhou]{Qi Zhou}
\address{Soochow University, Suzhou, P. R. China}
\email{zhouqi@suda.edu.cn}

\author[X. Zhu]{Xiaolin Zhu}
\address{Taiyuan Normal University, Taiyuan, P. R. China}
\email{XiaolinZhu@tynu.edu.cn}


\begin{document}

\begin{abstract}
We study two regularity problems on Hardy-type analytic tent spaces $\mathcal{AT}^p_{q,\alpha}$ on the unit disk: fractional integration and randomization of Taylor coefficients. For fractional integration, we characterize completely the boundedness and compactness of the Hadamard, Flett, and Riemann--Liouville operators between analytic tent spaces, and obtain parallel results for analytic Triebel--Lizorkin spaces. In particular, the case $t=0$ yields a complete solution to the corresponding embedding problem for analytic tent spaces. For randomization, we characterize completely when the random Taylor series $\mathcal{R}f$ belongs almost surely to an analytic tent space whenever $f\in \mathcal{AT}^p_{q,\alpha}$, and we also obtain the Triebel--Lizorkin counterpart. As part of the proof, we identify the random symbol space associated with $\mathcal{AT}^p_{q,\alpha}$ and solve the embedding problem from analytic tent spaces into mixed norm spaces. These results extend classical theorems of Hardy--Littlewood and Littlewood, as well as their later analogues for Bergman and mixed norm spaces.
\end{abstract}

\maketitle

\section{Introduction and main results}

\noindent
Let $\mathbb{D}$ and $\mathbb{T}$ denote the unit disk and the unit circle, respectively, and let $|d\xi|$ be the Lebesgue measure on $\mathbb{T}$. For $\gamma>1$ and $\xi\in\mathbb{T}$, the Kor\'anyi (admissible, non-tangential) approach region is defined by
\[
\Gamma_\gamma(\xi)=\left\{z\in\mathbb{D}: |1-z\bar{\xi}|<(\gamma/2)(1-|z|^2)\right\}.
\]
Throughout the paper we write $\Gamma(\xi)=\Gamma_2(\xi)$.

\medskip
\noindent
Let $0<p,q<\infty$ and let $\nu$ be a positive Borel measure on $\mathbb{D}$. The tent space $\mathcal{T}_{q,\nu}^p$ consists of $\nu$-measurable functions $f$ such that
\[
\|f\|_{\mathcal{T}_{q,\nu}^p}^p
=
\int_{\mathbb{T}}
\left(
\int_{\Gamma(\xi)} |f(z)|^q\, d\nu(z)
\right)^{p/q}
|d\xi|
<\infty.
\]
When $d\nu(z)=(1-|z|^2)^\alpha dA(z)$, we write $\mathcal{T}_{q,\nu}^p=\mathcal{T}_{q,\alpha}^p$. The Hardy-type analytic tent space is
\[
\mathcal{A}\mathcal{T}_{q,\alpha}^p
=
\mathcal{T}_{q,\alpha}^p \cap H(\mathbb{D}).
\]
These spaces belong to the holomorphic tent-space framework stemming from the tent spaces of Coifman, Meyer, and Stein \cite{Coi1985} and the analytic setting introduced by Cohn and Verbitsky \cite{Cohn2000}; see also \cite{Aguilar2023,Aguilar20231,Pelaez20152,perala2018}. Associated with this scale, we also consider the analytic Triebel--Lizorkin spaces $\mathcal{F}_{q,\alpha}^p$, defined by
\[
\mathcal{F}_{q,\alpha}^p
=
\left\{
f\in H(\mathbb{D}) :
(1-|z|)^{s-\alpha}\mathfrak{I}_{-s}^{\mathrm{F}}f
\in
\mathcal{T}_{q,-2}^p
\text{ for some } s>\alpha
\right\},
\]
where, for
\[
f(z)=\sum_{n=0}^\infty a_n z^n,
\]
the Flett fractional operator is given by
\[
\mathfrak{I}_t^{\mathrm{F}}f(z)
=
\sum_{n=0}^\infty (n+1)^{-t} a_n z^n.
\]
In the present paper, we shall also consider two further types of fractional operator, namely the Hadamard and Riemann--Liouville fractional  operators; see \eqref{Hadamard} and \eqref{Riemann}. 

\medskip
\noindent
The purpose of this paper is to study two regularity problems on analytic tent spaces. The first concerns the action of fractional integration operators, and the second concerns the effect of randomizing the Taylor coefficients. Although these two problems come from different directions, they are linked by a common question: 

\mednobf{Question.} 
\noindent\textit{
Within the scale of analytic tent spaces, how much regularity can be gained under a canonical operation such as fractional integration or randomization?}

\medno Analytic tent spaces provide a natural setting for this unified point of view, sitting between Hardy-type boundary behavior and Bergman-type area integrability, and interacting naturally with analytic Triebel--Lizorkin spaces. From this perspective, the two parts of the paper may be regarded as deterministic and probabilistic aspects of the same regularity-improvement phenomenon.

\bigskip
\subsection{Fractional integration and differentiation.}
In this subsection we study the boundedness and compactness of fractional integration and differentiation on Hardy-type analytic tent spaces and the related analytic Triebel--Lizorkin spaces. For
\[
f(z)=\sum_{n=0}^\infty a_n z^n \in H(\mathbb{D}),
\]
we consider the three standard fractional operators
\begin{equation}\label{Hadamard}
\mathfrak{I}_t^{\mathrm{H}}f(z)=\sum_{n=1}^\infty n^{-t} a_n z^n,
\qquad
\mathfrak{I}_t^{\mathrm{F}}f(z)=\sum_{n=0}^\infty (n+1)^{-t} a_n z^n, 
\end{equation}
and
\begin{equation}\label{Riemann}
\mathfrak{I}_t^{\mathrm{RL}}f(z)=\sum_{n=0}^\infty \frac{\Gamma(n+1)}{\Gamma(n+1+t)}\,a_n z^n,
\qquad t\in\mathbb{C}\setminus\{-1,-2,\dots\}, 
\end{equation}
with
\[
\mathfrak{I}_{-m}^{\mathrm{RL}}f=f^{(m)}, \qquad m\in\mathbb{N}.
\]
Thus, the Hadamard, Flett, and Riemann--Liouville cases can be treated within a single framework. In this paper we characterize completely the septuple $$(p, v, q, u; \alpha, \beta; t)
		\in (0, \infty)^4  \times (-2, \infty)^2 \times \mathbb{C}$$ such that the fractional integration
		operator $\mathfrak{I}_t$, of order $t \in \mathbb{C}$, is bounded (resp. compact) between two Hardy-type analytic tent spaces: $\mathfrak{I}_t: \mathcal{A}\mathcal{T}_{q,\alpha}^{p} \to
		\mathcal{A}\mathcal{T}_{u,\beta}^{v}.$
		We treat three types of fractional integration for $\mathfrak{I}_t$: Hadamard, Flett, and Riemann-Liouville.
		The corresponding results for the Hardy spaces $H^p(\mathbb{D})$ and the Bergman spaces $L_a^p(\mathbb{D},dA_\alpha)$ are due to Hardy and Littlewood in 1932 and Buckley-Koskela-Vukoti\'{c} in 1999, respectively.
		Our main result extend these results to analytic tent spaces and Triebel-Lizorkin spaces. Even the case $t=0$ is new and yields a complete solution of the  associated embedding problem.

\medskip
\noindent
\textbf{Boundedness.} Fractional integration on analytic function spaces has its classical starting point in the theorem of Hardy and Littlewood, who characterized the boundedness of the Riemann--Liouville operator on Hardy spaces $H^p(\mathbb{D})$ \cite{HL}. Closely related coefficient multipliers of Hadamard and Flett type were introduced by Hadamard \cite{jh1892} and Flett \cite{Flett1971,Flett1}, and the corresponding Hardy-space boundedness theory was subsequently developed by Flett and Kim \cite{Flett2,Kim}. In the Bergman-space setting, Buckley, Koskela, and Vukoti\'c obtained the analogous boundedness results for fractional integration between weighted Bergman spaces \cite{Buck}.

\medskip
\noindent
We begin by studying these operators on the broader scale of Hardy-type analytic tent spaces. More precisely, we seek to characterize when a given operator $\mathfrak{I}_t$, chosen from the three classes above, acts boundedly from
\[
\mathcal{A}\mathcal{T}_{q,\alpha}^p
\quad\text{to}\quad
\mathcal{A}\mathcal{T}_{u,\beta}^v.
\]
This extends the classical Hardy- and Bergman-space theory to a substantially richer setting and yields a unified treatment of the Hadamard, Flett, and Riemann--Liouville operators in the tent-space context.

\medskip
\noindent
Our first main result gives a complete criterion for the boundedness of fractional integration between Hardy-type analytic tent spaces.

\begin{theorem}\label{main thm}
Let $0<p,q,u,v<\infty$, $\alpha,\beta>-2$ and $t\in\mathbb{C}$. Then
\[
\mathfrak{I}_{t}:\mathcal{A}\mathcal{T}_{q,\alpha}^{p}\to \mathcal{A}\mathcal{T}_{u,\beta}^{v}
\]
is bounded if and only if one of the following conditions holds:
\begin{enumerate}
\item[($\romannumeral1$)] $p<v$ and
\[
\Re e\, t\ge \frac{1}{p}-\frac{1}{v}+\frac{\alpha+2}{q}-\frac{\beta+2}{u};
\]
\item[($\romannumeral2$)] $p\ge v$ and either
\[
\Re e\, t>\frac{\alpha+2}{q}-\frac{\beta+2}{u},
\]
or
\[
\Re e\, t=\frac{\alpha+2}{q}-\frac{\beta+2}{u}
\quad\text{and}\quad q\le u.
\]
\end{enumerate}
Here $\mathfrak{I}_{t}$ is one of $\mathfrak{I}_{t}^{\mathrm{H}}$, $\mathfrak{I}_{t}^{\mathrm{F}}$ and $\mathfrak{I}_{t}^{\mathrm{RL}}$.
\end{theorem}

\medskip
\noindent
Theorem \ref{main thm} allows the order $t$ to be complex. In particular, in the Riemann--Liouville case, negative integer orders recover the usual derivatives. Moreover, the criterion retains the familiar dichotomy between the cases $p<v$ and $p\ge v$, but now in the substantially richer setting of Hardy-type analytic tent spaces.

\medskip
\noindent
As an immediate consequence, setting $t=0$ yields a complete characterization of the embedding relation between Hardy-type analytic tent spaces. Previously, \cite{perala2018} established this embedding only in certain special cases, namely when $p\ge v$, $q\le u$, and
\[
\frac{\alpha+2}{q}=\frac{\beta+2}{u},
\]
and when $p<v$, $u=v$, and
\[
\frac{1}{p}-\frac{1}{v}+\frac{\alpha+2}{q}=\frac{\beta+2}{u}.
\]
The following corollary gives a complete answer for the full parameter range.

\begin{corollary}\label{tent embedding}
Let $0<p,q,u,v<\infty$ and $-2<\alpha,\beta<\infty$. Then
\[
\mathcal{A}\mathcal{T}_{q,\alpha}^{p}\subseteq \mathcal{A}\mathcal{T}_{u,\beta}^{v}
\]
if and only if one of the following conditions holds:
\begin{enumerate}
\item[($\romannumeral1$)] $p<v$ and
\[
\frac{1}{p}-\frac{1}{v}+\frac{\alpha+2}{q}-\frac{\beta+2}{u}\le 0;
\]
\item[($\romannumeral2$)] $p\ge v$ and either
\[
\frac{\alpha+2}{q}-\frac{\beta+2}{u}<0,
\]
or
\[
\frac{\alpha+2}{q}-\frac{\beta+2}{u}=0
\quad\text{and}\quad q\le u.
\]
\end{enumerate}
\end{corollary}

\medskip
\noindent
The same picture extends to analytic Triebel--Lizorkin spaces.

\begin{theorem}\label{mainthm3}
Let $0<p,q,u,v<\infty$, $\alpha,\beta\in\mathbb{R}$ and $t\in\mathbb{C}$. Then
\begin{equation}\label{Triebel bound formula}
\mathfrak{I}_{t}:\mathcal{F}_{q,\alpha}^{p}\to \mathcal{F}_{u,\beta}^{v}
\end{equation}
is bounded if and only if one of the following conditions holds:
\begin{enumerate}
\item[($\romannumeral1$)] $p<v$ and
\[
\Re e\, t\ge \frac{1}{p}-\frac{1}{v}-\alpha+\beta;
\]
\item[($\romannumeral2$)] $p\ge v$ and either
\[
\Re e\, t>-\alpha+\beta,
\]
or
\[
\Re e\, t=-\alpha+\beta
\quad\text{and}\quad q\le u.
\]
\end{enumerate}
Here $\mathfrak{I}_{t}$ is one of $\mathfrak{I}_{t}^{\mathrm{H}}$, $\mathfrak{I}_{t}^{\mathrm{F}}$ and $\mathfrak{I}_{t}^{\mathrm{RL}}$.
\end{theorem}

\medskip
\noindent
Thus we obtain a unified tent-space/Triebel--Lizorkin version of the Hardy--Littlewood theory. In particular, since $\mathcal{F}_{2,0}^{p}=H^{p}(\mathbb{D})$ and
\[
\mathcal{F}_{p,-\frac{\alpha+1}{p}}^{p}
=
L_{a}^{p}\left(\mathbb{D},dA_{\alpha}\right),
\]
Theorem \ref{mainthm3} contains both the classical Hardy-space results and the Bergman-space results as special cases.

\medskip
\noindent
Again, the case $t=0$ yields a complete embedding theorem for analytic Triebel--Lizorkin spaces.

\begin{corollary}\label{Triebel embedding}
Let $0<p,q,u,v<\infty$ and $\alpha,\beta\in\mathbb{R}$. Then
\[
\mathcal{F}_{q,\alpha}^{p}\subseteq \mathcal{F}_{u,\beta}^{v}
\]
if and only if one of the following conditions holds:
\begin{enumerate}
\item[($\romannumeral1$)] $p<v$ and
\[
\frac{1}{p}-\frac{1}{v}-\alpha+\beta\le 0;
\]
\item[($\romannumeral2$)] $p\ge v$ and either
\[
-\alpha+\beta<0,
\]
or
\[
-\alpha+\beta=0
\quad\text{and}\quad q\le u.
\]
\end{enumerate}
\end{corollary}

\medskip
\noindent
To the best of our knowledge, no previous partial results seem to have been available for this embedding problem in the setting of analytic Triebel--Lizorkin spaces.

\bigskip
\noindent
\textbf{Compactness.} We next turn to compactness. In contrast with the boundedness theory, compactness questions for fractional integration on analytic function spaces have been studied less extensively in the literature. In the present setting, the compactness criteria turn out to be the strict counterparts of the boundedness criteria. Thus the tent-space and Triebel--Lizorkin framework also yields a natural analogue of Sobolev-type compact embedding phenomena in the analytic category.

\medskip
\noindent
Our next result gives a complete characterization of compact fractional integration between Hardy-type analytic tent spaces.

\begin{theorem}\label{mainthm2}
Let $0<p,q,u,v<\infty$, $\alpha,\beta>-2$ and $t\in\mathbb{C}$. Then
\[
\mathfrak{I}_{t}:\mathcal{A}\mathcal{T}_{q,\alpha}^{p}\to \mathcal{A}\mathcal{T}_{u,\beta}^{v}
\]
is compact if and only if one of the following conditions holds:
\begin{enumerate}
\item[($\romannumeral1$)] $p<v$ and
\[
\Re e\, t>\frac{1}{p}-\frac{1}{v}+\frac{\alpha+2}{q}-\frac{\beta+2}{u};
\]
\item[($\romannumeral2$)] $p\ge v$ and
\[
\Re e\, t>\frac{\alpha+2}{q}-\frac{\beta+2}{u}.
\]
\end{enumerate}
Here $\mathfrak{I}_{t}$ is one of $\mathfrak{I}_{t}^{\mathrm{H}}$, $\mathfrak{I}_{t}^{\mathrm{F}}$ and $\mathfrak{I}_{t}^{\mathrm{RL}}$.
\end{theorem}

\medskip
\noindent
Thus, compared with Theorem \ref{main thm}, compactness is characterized by the strict version of the corresponding boundedness inequalities. In particular, the boundary cases that remain bounded are never compact.

\medskip
\noindent
The same phenomenon persists for analytic Triebel--Lizorkin spaces.

\begin{theorem}\label{mainthm4}
Let $0<p,q,u,v<\infty$, $\alpha,\beta\in\mathbb{R}$ and $t\in\mathbb{C}$. Then
\[
\mathfrak{I}_{t}:\mathcal{F}_{q,\alpha}^{p}\to \mathcal{F}_{u,\beta}^{v}
\]
is compact if and only if one of the following conditions holds:
\begin{enumerate}
\item[($\romannumeral1$)] $p<v$ and
\[
\Re e\, t>\frac{1}{p}-\frac{1}{v}-\alpha+\beta;
\]
\item[($\romannumeral2$)] $p\ge v$ and
\[
\Re e\, t>-\alpha+\beta.
\]
\end{enumerate}
Here $\mathfrak{I}_{t}$ is one of $\mathfrak{I}_{t}^{\mathrm{H}}$, $\mathfrak{I}_{t}^{\mathrm{F}}$ and $\mathfrak{I}_{t}^{\mathrm{RL}}$.
\end{theorem}

\medskip
\noindent
Theorem \ref{mainthm4} recovers, in particular, the compactness criteria for fractional integration between Hardy spaces and between weighted Bergman spaces. More generally, it also yields the compactness theory for fractional integration between Hardy and Bergman spaces in both directions. 

\medskip
\noindent
%
%

\subsection{Randomization, symbol spaces, and embeddings.}

\noindent
Now we come to the second part of this paper. A classical phenomenon in the theory of analytic series is that randomization of the Taylor coefficients may improve regularity almost surely. The prototype is Littlewood's theorem, which states that if $f\in H^2(\mathbb{D})$, then its Bernoulli randomization belongs almost surely to $H^p(\mathbb{D})$ for every $p>0$ \cite{Littlewood1930}. The same phenomenon is known for Steinhaus and Gaussian randomizations as well; see and \cite[p.~54]{Kahane1985}, \cite{Littlewood1926}, and \cite{Paley19302}. More recently, Cheng, Fang, and Liu \cite{IMRN} obtained a complete almost sure mapping theory for Bergman spaces and, more generally, for mixed norm spaces. In the present paper we study this problem on the broader scale of Hardy-type analytic tent spaces and the related analytic Triebel--Lizorkin spaces.

\medskip
\noindent
A random variable $X$ is called Bernoulli if
\[
\mathbb{P}(X=1)=\mathbb{P}(X=-1)=\frac12,
\]
Steinhaus if it is uniformly distributed on the unit circle, and by $N(0,1)$ we mean the law of a Gaussian variable with zero mean and unit variance. A standard random sequence $\{X_n\}_{n\ge0}$ refers to either a standard Bernoulli sequence, a standard Steinhaus sequence, or a standard Gaussian $N(0,1)$ sequence. For
\[
f(z)=\sum_{n=0}^{\infty} a_n z^n \in H(\mathbb{D}),
\]
we define its randomization by
\[
(\mathcal{R}f)(z)=\sum_{n=0}^{\infty} X_n a_n z^n.
\]
In this paper we obtain a complete characterization, in terms of the hexades
\[
(p,q,u,v,\alpha,\beta)\in(0,\infty)^4\times(-2,\infty)^2,
\]
of when randomization improves regularity almost surely on these scales of spaces. More precisely, our objective is to characterize when
\[
\mathcal{R}:\mathcal{A}\mathcal{T}_{q,\alpha}^{p}\hookrightarrow \mathcal{A}\mathcal{T}_{u,\beta}^{v},
\]
and likewise when
\[
\mathcal{R}:\mathcal{F}_{q,\alpha}^{p}\hookrightarrow \mathcal{F}_{u,\beta}^{v},
\]
where $\mathcal{R}:E\hookrightarrow F$ means that $\mathcal{R}f\in F$ almost surely for every $f\in E$.

\medskip
\noindent
Our main result for the probabilistic part is the following.

\begin{theorem}\label{mainthm}
Let $0<p,q,u,v<\infty$, $\alpha,\beta>-2$ and $\{X_n\}_{n\ge0}$ be a standard random sequence. Then
\[
\mathcal{R}:\mathcal{A}\mathcal{T}_{q,\alpha}^{p}\hookrightarrow \mathcal{A}\mathcal{T}_{u,\beta}^{v}
\]
if and only if one of the following conditions holds:
\begin{enumerate}
\item[($\romannumeral1$)] $0<p<2$, and
\begin{enumerate}
\item[($\romannumeral1$.1)] $p\le u$, and
\[
\frac1p+\frac{\alpha+2}{q}=\frac12+\frac{\beta+2}{u};
\]
\item[($\romannumeral1$.2)] 
\[
\frac1p+\frac{\alpha+2}{q}<\frac12+\frac{\beta+2}{u}.
\]
\end{enumerate}

\item[($\romannumeral2$)] $2\le p<\infty$, and
\begin{enumerate}
\item[($\romannumeral2$.1)] $u\ge \max\{2,q\}$, and
\[
\frac{\alpha+2}{q}=\frac{\beta+2}{u};
\]
\item[($\romannumeral2$.2)]
\[
\frac{\alpha+2}{q}<\frac{\beta+2}{u}.
\]
\end{enumerate}
\end{enumerate}
\end{theorem}

\medskip
\noindent
The same picture extends to analytic Triebel--Lizorkin spaces.

\begin{corollary}\label{main 3}
Let $0<p,q,u,v<\infty$, $\alpha,\beta\in\mathbb{R}$ and $\{X_n\}_{n\ge0}$ be a standard random sequence. Then
\[
\mathcal{R}:\mathcal{F}_{q,\alpha}^{p}\hookrightarrow \mathcal{F}_{u,\beta}^{v}
\]
if and only if one of the following conditions holds:
\begin{enumerate}
\item[($\romannumeral1$)] $0<p<2$, and
\begin{enumerate}
\item[($\romannumeral1$.1)] $p\le u$, and
\[
\frac1p-\alpha=\frac12-\beta;
\]
\item[($\romannumeral1$.2)]
\[
\frac1p-\alpha<\frac12-\beta.
\]
\end{enumerate}

\item[($\romannumeral2$)] $2\le p<\infty$, and
\begin{enumerate}
\item[($\romannumeral2$.1)] $u\ge \max\{2,q\}$, and
\[
\alpha=\beta;
\]
\item[($\romannumeral2$.2)]
\[
\alpha>\beta.
\]
\end{enumerate}
\end{enumerate}
\end{corollary}

\medskip
\noindent
Theorem \ref{mainthm} exhibits a sharp transition at $p=2$. It extends Littlewood's theorem and, at the same time, the Bergman-space and mixed-norm results of \cite{IMRN} to the substantially broader setting of Hardy-type analytic tent spaces. Corollary \ref{main 3} yields the corresponding extension to analytic Triebel--Lizorkin spaces. A central feature of our approach is that the randomization problem admits a deterministic reduction: it is driven by two ingredients, namely the identification of the relevant symbol spaces in Theorem \ref{symbol theorem} and a complete mixed-norm embedding theorem for analytic tent spaces in Theorem \ref{embedding theorem}. Both of these results are also of independent interest.

\medskip
\noindent
\textbf{Symbol spaces.} For a subspace $\mathcal{X}\subset H(\mathbb{D})$, define its \emph{symbol space} by
\[
\mathcal{X}_{\star}
=
\{f\in H(\mathbb{D}) : \mathbb{P}(\mathcal{R}f\in\mathcal{X})=1\}.
\]
By the Kolmogorov $0$-$1$ law \cite[Theorem 5.12, p.~86]{Cinlar}, one has
\[
\mathbb{P}(\mathcal{R}f\in\mathcal{X})\in\{0,1\}
\]
for every $f\in H(\mathbb{D})$ under mild assumptions on $\mathcal{X}$. Thus the notation $\mathcal{X}_{\star}$ is natural in the present setting. In particular, Littlewood's theorem may be reformulated as
\[
(H^p)_{\star}=H^2,
\qquad 0<p<\infty.
\]

\medskip
\noindent
Our first deterministic ingredient identifies the symbol space of $\mathcal{A}\mathcal{T}_{q,\alpha}^{p}$.

\begin{theorem}\label{symbol theorem}
Let $0<p,q<\infty$, $\alpha>-2$ and $\{X_n\}_{n\ge0}$ be a standard random sequence. Then
\[
(\mathcal{A}\mathcal{T}_{q,\alpha}^{p})_{\star}
=
H\!\left(2,q,\frac{\alpha+2}{q}\right).
\]
\end{theorem}

\medskip
\noindent
Here $H(p,q,\alpha)$ denotes the mixed norm space consisting of those $f\in H(\mathbb{D})$ such that
\[
\|f\|_{H(p,q,\alpha)}
=
\left(
\int_{0}^{1} M_p^q(f,r)\,(1-r)^{\alpha q-1}\,dr
\right)^{1/q}
<\infty.
\]
For background on mixed norm spaces we refer to \cite{Arevalo2015,Flett2,pm2019}.

\medskip
\noindent
Theorem \ref{symbol theorem} identifies the almost sure regularity threshold for randomization on Hardy-type analytic tent spaces. It is the tent-space analogue of Littlewood's theorem and also recovers the Bergman-space statement in \cite{IMRN}. Indeed, since
\[
L_a^p(dA)=\mathcal{A}\mathcal{T}_{p,-1}^{p},
\]
Theorem \ref{symbol theorem} yields
\[
\bigl(L_a^p(dA)\bigr)_{\star}
=
H\!\left(2,p,\frac1p\right),
\]
which is precisely the Bergman-space result of \cite{IMRN}.

\medskip
\noindent
The same point of view extends to analytic Triebel--Lizorkin spaces.

\begin{corollary}
Let $0<p,q<\infty$ and $\alpha\in\mathbb{R}$. Then
\[
(\mathcal{F}_{q,\alpha}^{p})_{\star}
=
H_{s}^{2,q,s-\alpha}.
\]
\end{corollary}

\medskip
\noindent
Here $H_{s}^{2,q,s-\alpha}$ denotes the Besov-type space of all $f\in H(\mathbb{D})$ such that
\[
\mathfrak{I}_{-s}^{\mathrm{F}}f\in H(2,q,s-\alpha),
\]
equipped with the norm
\[
\|f\|_{s}^{2,q,s-\alpha}
=
\|\mathfrak{I}_{-s}^{\mathrm{F}}f\|_{H(2,q,s-\alpha)}.
\]
This space is independent of the choice of $s$; see \cite[p.~188]{pm2019}. In particular, since $$\mathcal{F}_{2,0}^{p}=H^p(\mathbb{D}),$$ the above corollary recovers Littlewood's theorem.

\medno A natural question is whether Theorem \ref{symbol theorem}
	can be extended to cover the case $p=\infty$ or $q=\infty$. 
	If $q=\infty$, $p<\infty$ and $\alpha\in\mathbb{R}$, then by   the $L^p$-boundedness of the 
	admissible (non-tangential) maximal function (see \cite{Wang2020}), we have $$\mathcal{A}\mathcal{T}_{\infty,\alpha}^p=H^p.$$
	Hence,
	$$\left(\mathcal{A}\mathcal{T}_{\infty,\alpha}^p\right)_{\star}=H^2.$$ 
	The following appears to be a difficult (but highly captivating, at least to us) problem:
	
	\bignobf{Problem A.} Let $0<q<\infty$ and $\alpha\in \mathbb{R}$. Then, how to characterize 
	$(\mathcal{A}\mathcal{T}_{q, \alpha}^{\infty})_\star$  and $\left(\mathcal{F}_{q, \alpha}^{\infty}\right)_\star$? 
	
	\bigno In particular, one might observe that $\mathcal{F}_{2,0}^{\infty}=\mathrm{BMOA}$ \cite{pm2019}. 	
    
\medskip
\noindent
\textbf{Embedding.} We now turn to the second deterministic ingredient. By Theorem \ref{symbol theorem}, the almost sure mapping property
\[
\mathcal{R}:\mathcal{A}\mathcal{T}_{q,\alpha}^{p}\hookrightarrow \mathcal{A}\mathcal{T}_{u,\beta}^{v}
\]
is reduced to the inclusion
\[
\mathcal{A}\mathcal{T}_{q,\alpha}^{p}
\subset
H\!\left(2,u,\frac{\beta+2}{u}\right).
\]
More generally, this leads naturally to the problem of characterizing the embedding
\begin{equation}\label{E:embed}
\mathcal{A}\mathcal{T}_{q,\alpha}^{p}\subset H(u,v,\beta).
\end{equation}
Our second deterministic ingredient solves this embedding problem in full generality.

\begin{theorem}\label{embedding theorem}
Let $0<p,q,u,v<\infty$, $\alpha>-2$ and $\beta>0$. Then the embedding
\[
\mathcal{A}\mathcal{T}_{q,\alpha}^{p}\subset H(u,v,\beta)
\]
holds if and only if one of the following conditions is satisfied:
\begin{enumerate}
\item[($\romannumeral1$)] $0<p<u$, and
\begin{enumerate}
\item[($\romannumeral1$.1)] $p\leq v$, and
\[
\frac{1}{p}+\frac{\alpha+2}{q}=\frac{1}{u}+\beta;
\]
\item[($\romannumeral1$.2)]
\[
\frac{1}{p}+\frac{\alpha+2}{q}<\frac{1}{u}+\beta.
\]
\end{enumerate}

\item[($\romannumeral2$)] $p=u$, and
\begin{enumerate}
\item[($\romannumeral2$.1)] $v\geq \max\{p,q\}$, and
\[
\frac{\alpha+2}{q}=\beta;
\]
\item[($\romannumeral2$.2)]
\[
\frac{\alpha+2}{q}<\beta.
\]
\end{enumerate}

\item[($\romannumeral3$)] $p>u$, and
\begin{enumerate}
\item[($\romannumeral3$.1)] $v\geq \max\{u,q\}$, and
\[
\frac{\alpha+2}{q}=\beta;
\]
\item[($\romannumeral3$.2)]
\[
\frac{\alpha+2}{q}<\beta.
\]
\end{enumerate}
\end{enumerate}
\end{theorem}

\medskip
\noindent
Theorem \ref{embedding theorem} gives a complete characterization of the embedding of analytic tent spaces into mixed norm spaces. Combined with Theorem \ref{symbol theorem}, it yields Theorem \ref{mainthm}. In this way, the probabilistic mapping problem is reduced entirely to two deterministic statements, each of which may be viewed as a natural structural result about analytic tent spaces in its own right.

\subsection{Methodology.}

\medno 
For the fractional operator part, to the best of our knowledge, no analogous boundedness theory for fractional integration on analytic tent spaces or analytic Triebel--Lizorkin spaces is currently available in the literature. Our approach is therefore necessarily different from those used in the existing works. We now briefly describe the main ingredients of the proof.

\medno The first ingredient is the extension to complex orders. The key point here is the invertibility of fractional integration of purely imaginary order. Combining a crucial estimate (Lemma \ref{V_n frac estimate}) due to Pavlović \cite{Pavlovic2013} with the dyadic projection machinery (Definition \ref{Vn definition} and Lemma \ref{Tent V_n biaoda}) and the Cesàro maximal operator (Lemma \ref{cesaro max}), we establish the boundedness and invertibility of $\mathfrak{I}_{i\delta}$ for $\delta\in\mathbb{R}$ on Hardy-type analytic tent spaces.

\medno A second ingredient is the passage between the three models of fractional integration considered in this paper. Here the Equivalence Lemma (Lemma \ref{equavilence}) from \cite{Guo3} allows us to move freely among the Hadamard, Flett, and Riemann--Liouville operators.

\medno A further ingredient is the discrete tent-space framework developed in \cite{Aguilar2023, Ars1999, Coi1985, Jev1996, Lueck1991, Wang2020}. To this end, we work with the tent space of sequences associated with an $r$-lattice $Z=\{a_k\}\subset\mathbb{D}$, where $r\in(0,1)$; see Theorem 2.23 in \cite{zhu2005}. More precisely, for $0<p,q<\infty$, we consider the space $T_q^p(Z)$ consisting of all sequences $\{\lambda_k\}$ such that
$$
\|\{\lambda_k\}\|_{T_q^p(Z)}
:=
\left(\int_{\mathbb{T}}
\left(\sum_{a_k\in\Gamma(\xi)}|\lambda_k|^q\right)^{p/q}
|d\xi|\right)^{1/p}
<\infty.
$$
The factorization method (Lemma \ref{Tent fraction}) together with the duality principles (Lemma \ref{Tent dual p>1} and Lemma \ref{duality norm}) is then used to complete the proof of Theorem \ref{main thm}. A central role is played by the following discrete characterization: for $0<p,q<\infty$, $\alpha>-2$, and an $r$-lattice $Z=\{a_k\}\subset\mathbb{D}$,
$$
\|f\|_{\mathcal{A}\mathcal{T}_{q,\alpha}^p}
\asymp
\left\|\left\{|f(a_k)|(1-|a_k|)^{\frac{\alpha+2}{q}}\right\}\right\|_{T_q^p(Z)}.
$$

\medskip
\noindent
For the randomization part, for the sufficiency part of ($\romannumeral1$) in Theorem \ref{embedding theorem}, we use fractional integration operators (Theorem 6 in \cite{Guo2}) together with fractional $g$-functions (Theorem 10.10 in \cite{pm2019}). We also employ an extrapolation argument: we first prove the case $q=2k$ with $k\ge1$, and then pass to the full range of parameters. For the necessity part, the test function
$$
f_{b,c}(z)=\frac{1}{(1-z)^b}\left(\log \frac{e}{1-z}\right)^{-c}
$$
from Lemma \ref{Fbc mix} and Lemma \ref{Fbc tent} is used.

\medskip
\noindent
While the sufficiency parts of ($\romannumeral2$) and ($\romannumeral3$) in Theorem \ref{embedding theorem} require several additional tools, the necessity arguments are more delicate. The necessity of part ($\romannumeral2$) is obtained by testing the lacunary series
$$
f(z)=\sum_{n=1}^{\infty} a_n z^{2^n-1}.
$$
For part ($\romannumeral3$), when $0<u<p<\infty$, we do not know how to construct an explicit counterexample showing that
\begin{align}\label{counter case}
	\mathcal{A}\mathcal{T}_{q,\alpha}^{p}\nsubseteq H(u,v,\beta).
\end{align}
for $v\ge q$ and $v<u$. At this point, the results in \cite{Guo2}, which characterize the boundedness of fractional integration from Hardy spaces to mixed norm spaces, provide the needed input. Using these fractional-integration techniques, we first prove \eqref{counter case} in the case
$$
q=2/k, \qquad k\ge1.
$$
We then deduce \eqref{counter case} for all $v\ge q$ and $v<u$ by means of the embedding corollary (Corollary \ref{tent embedding}) from \cite{perala2018}. We conclude the paper with an open problem asking whether one can avoid the above fractional-integration argument and the $2/k$ reduction in a more direct approach to this embedding problem.

\subsection{Organization of the paper.}

\noindent
Section~\ref{preliminary} gathers the preliminary material used throughout the paper. In particular, it contains the fractional integration tools, the radial and discrete descriptions of analytic tent spaces, the sequence tent-space machinery, and the $p$-Banach-space estimates that will later be used in the probabilistic part.

\medskip
\noindent
The fractional part is treated first. Section~\ref{proof of bound sec} establishes the boundedness theorem for Hardy-type analytic tent spaces, namely Theorem~\ref{main thm}, and also yields the analytic Triebel--Lizorkin counterpart, Theorem~\ref{mainthm3}. Section~\ref{proof of compact sec} is devoted to compactness: it proves Theorem~\ref{mainthm2}, and the same reduction as in the boundedness part gives the Triebel--Lizorkin counterpart Theorem~\ref{mainthm4}.

\medskip
\noindent
The probabilistic part is carried out in the remaining sections. Section~\ref{sec:proof-thm-sym} proves the symbol-space theorem, Theorem~\ref{symbol theorem}. Section~\ref{sec:proof-thm-emdeding} proves the mixed norm embedding theorem, Theorem~\ref{embedding theorem}. Combined with the reduction developed in Section~\ref{preliminary}, these two results yield the almost sure mapping theorem on analytic tent spaces, Theorem~\ref{mainthm}, and its analytic Triebel--Lizorkin counterpart, Corollary~\ref{main 3}. The paper concludes with a remark and an open problem concerning the mixed norm embedding theorem.

\bigskip
\subsection{Notation and conventions.}

\noindent
Let $H(\mathbb{D})$ denote the space of analytic functions on the unit disk. For $0<p<\infty$, the Hardy space $H^p(\mathbb{D})$ consists of those $f\in H(\mathbb{D})$ such that
\[
\|f\|_{H^p(\mathbb{D})}
=
\sup_{0<r<1} M_p(r,f)
<\infty,
\]
where
\[
M_p(r,f)
=
\left(
\frac{1}{2\pi}\int_{\mathbb{T}} |f(r\xi)|^p\,|d\xi|
\right)^{1/p}.
\]

\medskip
\noindent
For $\alpha>-1$, let $dA_\alpha$ denote the weighted area measure
\[
dA_\alpha(z)
=
\frac{\alpha+1}{\pi}(1-|z|^2)^\alpha\,dA(z),
\]
where $dA$ is the Euclidean area measure on $\mathbb{D}$. The weighted Bergman space is
\[
L_a^p(\mathbb{D},dA_\alpha)
=
L^p(\mathbb{D},dA_\alpha)\cap H(\mathbb{D}).
\]
For convenience, we shall often write $H^p$ in place of $H^p(\mathbb{D})$ and $L_a^p(dA_\alpha)$ in place of $L_a^p(\mathbb{D},dA_\alpha)$.

\medskip
\noindent
Given two nonnegative quantities $A$ and $B$, we write $A\lesssim B$ if there exists a positive constant $C$, independent of the relevant variables, such that $A\leq C B$. The notation $A\gtrsim B$ is defined analogously, and $A\asymp B$ means that both $A\lesssim B$ and $A\gtrsim B$ hold. For $1<p<\infty$, we write $p'=p/(p-1)$ for the H\"older conjugate of $p$. Finally, the symbol ``$\Leftrightarrow$'' means ``if and only if''.

\section{Preliminary Issues}\label{preliminary}

\medno To clear up the ground, in this section we collect the necessary preliminary material in the following four aspects:
\begin{enumerate}
	\item testing functions, including the power function, the lacunary series, and the logarithmically perturbed test functions;
	\item fractional integration tools, including the invertibility lemma for fractional integration operators of purely imaginary orders and the equivalence lemma for three types of fractional integration;
	\item discrete tent-space tools, including the discretization of tent spaces, the factorization lemma, and the duality results;
	\item $p$-Banach space techniques, which rely on the framework developed in \cite{IMRN}.
\end{enumerate}

\subsection{Fractional integration tools}

\medno Here we deal with the fractional integration tools needed later.

\begin{lemma}\label{invert}
	If $0<p,q<\infty$, $\alpha>-2$ and $\sigma\in\mathbb{R},$ then
	$\mathfrak{I}_{i\sigma}: \mathcal{A}\mathcal{T}_{q,\alpha}^{p} \to \mathcal{A}\mathcal{T}_{q,\alpha}^{p}$ is bounded and invertible, where $\mathfrak{I}_{i\sigma}$ is one of $\mathfrak{I}_{i\sigma}^{\mathrm{H}}$, $\mathfrak{I}_{i\sigma}^{\mathrm{F}}$ and $\mathfrak{I}_{i\sigma}^{\mathrm{RL}}$.
\end{lemma}

\smallno To prove the invertibility lemma, we shall need the following.

\begin{definition}\label{Vn definition}
	Let $V_n$ be polynomials such that, for all $p \in(0, \infty]$,
	$$
	\begin{aligned}
			\operatorname{supp}\left(\widehat{V}_n\right)  \subseteq\left(2^{n-1}, 2^{n+1}\right),\ \text { for } n \geq 1, \ \text { and } \ \operatorname{supp}\left(\widehat{V}_0\right) \subset[0,2);
	\end{aligned}
	$$
	$$
	\begin{aligned}
	 f(z)  =\sum_{n=0}^{\infty} V_n * f(z), \ \ f \in H(\mathbb{D}), \ \ z \in \mathbb{D};
	\end{aligned}
	$$
	$$
	\begin{aligned}
	\left\|V_n * f\right\|_{H^p}  \leq C\|f\|_{H^p}, \ \ f \in H^p ;\ \ \text { and } \ \left\|V_n\right\|_{H^p}  =2^{n(1-1 / p)}.
	\end{aligned}
	$$
\end{definition}

\begin{lemma}[\cite{Oswald1983}]\label{cesaro max}
	Let $0<p,q<\infty$ and $m>\max\{0,\frac{1}{p}-1,\frac{1}{q}-1\}$. If $\{f_{j}\}_{j=0}^{\infty}$ is a sequence in $H^{q}$, then
	$$
	\int_{\mathbb{T}}
	\left(
	\sum_{j=0}^{\infty}\sigma_{*}^{m}f_{j}(\xi)^{q}\right)
	^{\frac{p}{q}}|d\xi|
	\leq C_{p,q,m}
	\int_{\mathbb{T}}\left(
	\sum_{j=0}^{\infty}|f_{j}(\xi)|^{q}\right)
	^{\frac{p}{q}}|d\xi|.
	$$
\end{lemma}

\medno For more information on the Ces\`{a}ro maximal operators $\sigma_{*}^{m}$, one may refer to \cite{Pavlovic2013}.

\begin{lemma}[\cite{Pavlovic2013}]\label{V_n frac estimate}
	If $\gamma,\sigma\in\mathbb{R}$ and $f\in H(\mathbb{D})$, then
	$$
	\big|
	\mathfrak{I}_{\gamma+i\sigma}^{\mathrm{H}}V_{n}*f\big|
	\lesssim 2^{n\gamma}
	\sigma_{*}^{m}(V_{n}*f)
	$$
	and
	$$
	\big|V_{n}*f\big|
	\lesssim  2^{-n\gamma}
	\sigma_{*}^{m}(\mathfrak{I}_{\gamma+i\sigma}^{\mathrm{H}}V_{n}*f).
	$$
\end{lemma}

\begin{lemma}[\cite{pm2019}]\label{Tent V_n biaoda}
	Let $0<p,q<\infty$ and $\alpha>-2.$ Then for $f\in H(\mathbb{D})$,
	$$
	\int_{\mathbb{T}}\left(
	\int_{0}^{1}|f(r\xi)|^{q}(1-r)^{\alpha+1}dr\right)^{\frac{p}{q}}|d\xi|
	\asymp
	\int_{\mathbb{T}}\left(
	\sum_{n=0}^{\infty}2^{-n(\alpha+2)}|V_{n}*f(\xi)|\right)^{\frac{p}{q}}|d\xi|.
	$$
\end{lemma}

\smallno We now proceed with the proof of Lemma \ref{invert}.

\begin{proof}[Proof of Lemma \ref{invert}]
	It suffices to prove, for $\sigma\in\mathbb{R}$, $0<p,q<\infty$ and $\alpha>-2,$
	$$
	||\mathfrak{I}_{i\sigma}^{\mathrm{H}}f||_{\mathcal{A}\mathcal{T}_{q,\alpha}^{p}}
	\asymp
	||f||_{\mathcal{A}\mathcal{T}_{q,\alpha}^{p}},
	$$
	where $f\in \mathcal{A}\mathcal{T}_{q,\alpha}^{p}$ and $f(0)=0.$
	By Lemma \ref{cesaro max} and Lemma \ref{V_n frac estimate}, followed by Lemma \ref{Tent V_n biaoda}, we have
	\begin{align*}
		||\mathfrak{I}_{i\sigma}^{\mathrm{H}}f||_{\mathcal{A}\mathcal{T}_{q,\alpha}^{p}}^p
		&\asymp
		\int_{\mathbb{T}}\left(
		\sum_{n=0}^{\infty}2^{-n(\alpha+2)}
		|V_{n}*\mathfrak{I}_{i\sigma}^{\mathrm{H}} f(\xi)|^{q}\right)^{\frac{p}{q}}|d\xi|\\
		&\lesssim
		\int_{\mathbb{T}}\left(
		\sum_{n=0}^{\infty}2^{-n(\alpha+2)}
		|\sigma_{m}^{*}V_{n}*f(\xi)|^{q}
		\right)^{\frac{p}{q}}|d\xi|\\
		&\lesssim
		\int_{\mathbb{T}}\left(
		\sum_{n=0}^{\infty}2^{-n(\alpha+2)}
		|V_{n}*f(\xi)|^{q}
		\right)^{\frac{p}{q}}|d\xi|.
	\end{align*}
	The converse direction follows from the fact $\mathfrak{I}_{i\sigma}^{\mathrm{H}}\mathfrak{I}_{-i\sigma}^{\mathrm{H}}=I$, and the proof is complete now.
\end{proof}

\begin{lemma}[\cite{Guo3}
]\label{equavilence}
	Let $0<p,q<\infty$, $\alpha>-2$ and $t\in\mathbb{C}.$ Then for $f\in H(\mathbb{D})$,
	$$
	\mathfrak{I}_{t}^{\mathrm{H}}f\in \mathcal{A}\mathcal{T}_{q,\alpha}^{p}
	\Leftrightarrow
	\mathfrak{I}_{t}^{\mathrm{F}}f\in \mathcal{A}\mathcal{T}_{q,\alpha}^{p}
	\Leftrightarrow
	\mathfrak{I}_{t}^{\mathrm{RL}}f\in \mathcal{A}\mathcal{T}_{q,\alpha}^{p}.
	$$
\end{lemma}

\begin{lemma}[\cite{Avetisyan2012}]\label{t func trans}
		Let $b>0$, $t\in\mathbb{R}$ and $b-t>0$. Then for all $z\in\mathbb{D}$ $$
		\Big|
		\mathfrak{I}_{t}^{\mathrm{RL}}\frac{1}{(1-z)^{b}}\Big|
		\asymp \Big|
		\frac{1}{(1-z)^{b-t}}\Big|.$$
	\end{lemma}

	\begin{lemma}[\cite{Guo3, pm2019}]\label{tent frac trans}
		Let $0<p<\infty$, $0<q < \infty$ and $f \in H(\mathbb{D})$. If $t \in \mathbb{R}$, $\alpha>-2$ and $\alpha-t q>-2$, then $$f \in \mathcal{A}\mathcal{T}_{q, \alpha}^p \Leftrightarrow \mathfrak{I}_t f \in \mathcal{A}\mathcal{T}_{q, \alpha-t q}^p,$$ where $\mathfrak{I}_t$ is one of $\mathfrak{I}_t^\mathrm{H}$, $\mathfrak{I}_t^\mathrm{F}$ and $\mathfrak{I}_t^{\mathrm{RL}}$.
	\end{lemma}

	\begin{lemma}[\cite{Guo1}]\label{Fbc trans}
		Let $0<b, c<\infty$, $t \in \mathbb{R}$, $b-t>0$ and $$f_{b, c}(z)=\frac{1}{(1-z)^b}\left(\log \frac{e}{1-z}\right)^{-c}.$$ Then for all $z \in \mathbb{D}$,
		$$
		\left|\mathfrak{I}_t^{\mathrm{RL}} f_{b, c}(z)\right| \asymp\left|f_{b-t, c}(z)\right|.
		$$
	\end{lemma}

\begin{lemma}[ \cite{Guo3, pm2019}]\label{Fractional Calderon}
		If $p, \alpha \in (0, \infty)$,  $f \in H(\mathbb{D})$ with $f(0)=0$, then
		\begin{align}\label{g func equa}
			\|f\|_{H^p}^p \asymp \int_\mathbb{T}\left(\int_{\Gamma(\xi)}\left|\mathfrak{I}_{-\alpha}^\mathrm{F}f(z)\right|^2\left(1-|z|^2\right)^{2 \alpha-2} d A(z)\right)^{p/2}|d \xi|.
		\end{align}
	\end{lemma}

\subsection{Radial descriptions and testing functions}
\medno We first record a radial description of analytic tent spaces, which will be used  to connect tent spaces with mixed norm spaces.
\begin{lemma}[{\cite{pm2019}}]
\label{tent-radial-description}
Let $0<p, q<\infty$, $\alpha>-2$ and $f \in H(\mathbb{D})$. Then
\[
\int_{\mathbb{T}}
\left(
\int_0^1 |f(r\xi)|^q (1-r)^{\alpha+1}\,dr
\right)^{\frac{p}{q}}
|d\xi|
\asymp
\int_{\mathbb{T}}
\left(
\int_{\Gamma(\xi)} |f(z)|^q (1-|z|)^\alpha\,dA(z)
\right)^{\frac{p}{q}}
|d\xi|.
\]
\end{lemma}

\medno The following classes of examples will serve as test functions in this paper.

\begin{lemma}\label{t function}
	Let $0<p,q < \infty$ and $\alpha>-2$. Then
	for $\gamma \in \mathbb{R}$, the function
	$$
	f(z)=\frac{1}{(1-z)^{\gamma}}
	$$
	belongs to $\mathcal{A}\mathcal{T}_{q,\alpha}^p$ if and only if $$\gamma<\frac{1}{p}+\frac{\alpha+2}{q}.$$
\end{lemma}
\begin{proof}
    The proof of Lemma \ref{t function} is similar to that of Lemma \ref{Fbc tent} below and is therefore omitted.
\end{proof}
\medno We shall also need two facts on Hadamard's lacunary series.

\begin{lemma}[\cite{pm2019}]\label{tent lacunary}
	Let
	\begin{align}\label{lacunary formula}
		f(z)=\sum_{n=1}^{\infty}a_{n}z^{2^{n}-1}.
	\end{align}
	and let $0<p,q<\infty$, $\alpha>-2.$ Then $f\in \mathcal{A}\mathcal{T}_{q,\alpha}^{p}$ if and only if
	$$
	\left\{ 2^{-n\frac{\alpha+2}{q}}a_{n}
	\right\}_{n=0}^{\infty}\in \ell^{q}.
	$$
\end{lemma}

\begin{proof}
	It can be deduced from Theorem 10.27 in \cite{pm2019}.
\end{proof}

\begin{lemma}[\cite{MM1}]\label{mix lacunary}
	Let $0<p, q < \infty$ and $0<\alpha<\infty$. Then the function (\ref{lacunary formula}) belongs to $H(p, q, \alpha)$ if and only if $$\left\{2^{-n\alpha} a_{n} \right\}_{n=0}^\infty\in \ell^{q}.$$
\end{lemma}

\medno The corresponding facts for the second test function $f_{b, c}(z)$ are as follows, with the first one due to Avetisyan.

\begin{lemma}[\cite{Avetisyan2012}]\label{Fbc mix}
	Let $0<p,q< \infty$, $b,c>0$ and $\alpha>0$. Let
	\begin{align}\label{Fbc}
		f_{b,c}(z)=\frac{1}{(1-z)^{b}}
		\left(\log\frac{e}{1-z}\right)^{-c}.
	\end{align}
	Then  $f \in H(p,q,\alpha)$ if and only if
	$$
	\begin{cases}
		b=\alpha+\frac{1}{p},\ c>\frac{1}{q};\ or\\
		b<\alpha+\frac{1}{p}.
	\end{cases}
	$$
\end{lemma}

\begin{lemma}\label{Fbc tent}
	Let $0<p,q<\infty$,  $b,c>0$ and $\alpha>-2$. Then the function (\ref{Fbc})
	belongs to $\mathcal{A}\mathcal{T}_{q,\alpha}^p$ if and only if
	$$
	\begin{cases}
		b=\frac{\alpha+2}{q}+\frac{1}{p},\ c>\frac{1}{p};\ or\\
		b<\frac{\alpha+2}{q}+\frac{1}{p}.
	\end{cases}
	$$
\end{lemma}

\medno The rest of this subsection is devoted to proving Lemma \ref{Fbc tent}, incorporating several extra facts that will be introduced in due course.

	\noindent In order to prove Lemma \ref{Fbc tent}, We shall need the following lemma.

	\begin{lemma}\label{J es}
	Let \(\alpha,\beta\in\mathbb R\), and define
	\[
	J_{\alpha,\beta}(\theta)
	:=
	\int_0^1 |1-re^{i\theta}|^{-\alpha-1}
	\left|\log\frac{e}{1-re^{i\theta}}\right|^{-\beta}\,dr,
	\qquad 0<|\theta|<\pi.
	\]
	Then
	\[
	J_{\alpha,\beta}(\theta)\asymp
	\begin{cases}
	|\theta|^{-\alpha}\Bigl(\log \dfrac{e}{|\theta|}\Bigr)^{-\beta}, & \alpha>0,\\[1ex]
	1, & \alpha<0,
	\end{cases}
	\]
	and, when \(\alpha=0\),
	\[
	J_{0,\beta}(\theta)\asymp
	\begin{cases}
	\Bigl(\log \dfrac{e}{|\theta|}\Bigr)^{1-\beta}, & \beta<1,\\[1ex]
	\log\log \dfrac{e}{|\theta|}, & \beta=1,\\[1ex]
	1, & \beta>1.
	\end{cases}
	\]
	The implicit constants depend only on \(\alpha\) and \(\beta\).
	\end{lemma}

	\begin{proof}
	It suffices to consider \(0<|\theta|<\frac12\).
	Set
	\[
	E:=\left\{re^{i\theta}\in\mathbb D:\frac{9}{10}<r<1,\ |\theta|<\frac12\right\}.
	\]
	By standard estimates on \(E\), we have
	\[
	J_{\alpha,\beta}(\theta)
	\asymp
	1+\int_{9/10}^1
	\frac{dr}{(1-r+|\theta|)^{\alpha+1}
	\bigl(\log \frac{e}{1-r+|\theta|}\bigr)^\beta }.
	\]
	Setting \(x=1-r+|\theta|\), we get
	\[
	J_{\alpha,\beta}(\theta)
	\asymp
	1+\int_{|\theta|}^{1/10}
	\frac{dx}{x^{\alpha+1}\bigl(\log \frac{e}{x}\bigr)^\beta }.
	\]
	If \(\alpha>0\), the last integral is dominated by the lower limit, and hence
	\[
	J_{\alpha,\beta}(\theta)
	\asymp
	|\theta|^{-\alpha}\Bigl(\log \frac{e}{|\theta|}\Bigr)^{-\beta}.
	\]
	If \(\alpha<0\), the integral converges uniformly near \(x=0\), so
	\[
	J_{\alpha,\beta}(\theta)\asymp 1.
	\]
	If \(\alpha=0\), then
	\[
	\int_{|\theta|}^{1/10}\frac{dx}{x\bigl(\log \frac{e}{x}\bigr)^\beta}
	\asymp
	\begin{cases}
	\Bigl(\log \dfrac{e}{|\theta|}\Bigr)^{1-\beta}, & \beta<1,\\[1ex]
	\log\log \dfrac{e}{|\theta|}, & \beta=1,\\[1ex]
	1, & \beta>1.
	\end{cases}
	\]
	This proves the lemma.
	\end{proof}

	\smallno
	\begin{proof} Now we continue the proof of Lemma \ref{Fbc tent}. By Lemma \ref{tent frac trans}, $f_{b,c} \in A T_{q, \alpha}^p$
	if and only if
	$$
	\mathfrak{I}_{\frac{\alpha+1}{q}}^{\mathrm{RL}} f(z) \in \mathcal{A}\mathcal{T}_{q,-1}^p.
	$$
	Here we assume $b>\frac{\alpha+2}{q}$.
	By Lemma \ref{Fbc trans},  it suffices to consider
	$$
	\frac{1}{(1-z)^{b-\frac{\alpha+1}{q}}}\left(\log \frac{e}{1-z}\right)^{-c} \in \mathcal{A}\mathcal{T}_{q,-1}^p.
	$$
	By Lemma \ref{J es}, we have,
\[
\|f_{b,c}\|_{\mathcal{A}\mathcal{T}_{q,\alpha}^p}^p
\asymp
\int_0^\pi
\theta^{\left(\frac{\alpha+2}{q}-b\right)p}
\left(\log \frac{e}{\theta}\right)^{-cp}\,d\theta.
\]
The proof of Lemma \ref{Fbc tent} is complete now.
\end{proof}

\subsection{Discrete tent-space tools}

\medno We need several lemmas on the discrete tent spaces in the proof of the main theorems.

\begin{definition}\label{discerte tent}
	Let $Z=\left\{z_n\right\}$ be an $r$-lattice and $0<p, q<+\infty$. We say that $\left\{\lambda_n\right\} \in T_p^q(Z)$ if
	$$
	\left\|\left\{\lambda_n\right\}\right\|_{T_p^q(Z)}:=\left(\int_{\mathbb{T}}\left(\sum_{z_n \in \Gamma(\xi)}\left|\lambda_n\right|^p\right)^{q / p}|d \xi|\right)^{1 / q}<+\infty.
	$$
\end{definition}

\begin{lemma}\label{Tent discrete}
	Let $0<p, q< \infty$ and $Z=\left\{a_k\right\}$ be an $r$-lattice. If $f \in \mathcal{A}\mathcal{T}_{q,\alpha}^p$, then
	$$
	\|f\|_{\mathcal{A}\mathcal{T}_{q,\alpha}^p} \asymp\left\|\left\{|f(a_k)|(1-|a_k|)^{\frac{\alpha+2}{q}}\right\}\right\|_{T_q^p(Z)}.
	$$
\end{lemma}

\begin{proof}
	By Lemma 2 in \cite{Wang2020}, it suffices to prove
	$$\left\|\left\{|f(a_k)|(1-|a_k|)^{\frac{\alpha+2}{q}}\right\}\right\|_{T_q^p(Z)} \lesssim \|f\|_{\mathcal{A}\mathcal{T}_{q,\alpha}^p}.$$
	We can choose $\widetilde{\Gamma}(\xi)$ such that
	\begin{align}\label{Gamma inc}
		\bigcap_{a_k\in\Gamma(\xi)} D(a_k,r)\subseteq \widetilde{\Gamma}(\xi),
	\end{align}
	where $D\left(a_k, r\right)$ are hyperbolic disks that correspond to the $r$-lattice.
	Then by (\ref{Gamma inc}) and properties of $r$-lattices,
	$$
	\begin{aligned}
		& \int_\mathbb{T}\left(\sum_{a_{k} \in \Gamma(\xi)}\left|f\left(a_k\right)\right|^q\left(1-\left|a_k\right|^2\right)^{\alpha+2}\right)^{\frac{p}{q}}|d \xi| \\
		&\qquad \lesssim \int_\mathbb{T}\left(\sum_{a_k \in{\widetilde{\Gamma} (\xi)}} \int_{D\left(a_{k}, r\right)}  |f(z)|^q\left(1-|z|^2\right)^\alpha d A(z)\right)^{\frac{p}{q}}|d \xi| \\
		&\qquad \lesssim \int_\mathbb{T}\left(\int_{\widetilde{\widetilde{\Gamma}}(\xi)}|f(z)|^q\left(1-|z|^2\right)^\alpha d A(z)\right)^{\frac{p}{q}} |d \xi|  \\
		&\qquad \asymp \int_\mathbb{T}\left(\int_{\Gamma(\xi)}\left|f(z)\right|^q\left(1-|z|^2\right)^\alpha d A(z)\right)^{\frac{p}{q}}|d \xi|.
	\end{aligned}
	$$
	The last step follows from the fact that different apertures define the same tent space (with equivalent quasinorms).
\end{proof}

\begin{lemma}[\cite{Coi1985, Wang2020}]\label{Tent fraction}
	Let $0<p, q<+\infty$ and $Z=\left\{a_k\right\}$ be an $r$-lattice. If $p \leq p_1, p_2 <+\infty$, $q \leq q_1, q_2 <+\infty$ and satisfy $$\frac{1}{p}=\frac{1}{p_1}+\frac{1}{p_2},$$ and $$\frac{1}{q}=\frac{1}{q_1}+\frac{1}{q_2},$$ then
	$$
	T_p^q(Z)=T_{p_1}^{q_1}(Z) \cdot T_{p_2}^{q_2}(Z).
	$$
\end{lemma}

\begin{lemma}[\cite{Ars1999, Jev1996, Lueck1991}]\label{Tent dual p>1}
	Let $Z=\left\{z_n\right\}$ be an $r$-lattice and $1 < p<+\infty$, $1<q<+\infty$. Then $$\left(T_p^q(Z)\right)^* \cong T_{p^{\prime}}^{q^{\prime}}(Z),$$ where $\frac{1}{p}+\frac{1}{p^{\prime}}=1$ and $\frac{1}{q}+\frac{1}{q^{\prime}}=1$. The isomorphism between $\left(T_p^q(Z)\right)^*$ and $T_{p^{\prime}}^{q^{\prime}}(Z)$ is given by the operator
	$$
	\left\{b_k\right\} \mapsto\left\langle\cdot,\left\{b_k\right\}\right\rangle,
	$$
	where $\left\langle\cdot,\left\{b_k\right\}\right\rangle$ is defined by
	$$
	\left\langle\left\{a_k\right\},\left\{b_k\right\}\right\rangle=\sum_k a_k b_k\left(1-\left|z_k\right|\right), \quad\left\{a_k\right\} \in T_p^q(Z).
	$$
	In fact,
	$$
	\left\|\left\{b_k\right\}\right\|_{T_{p^{\prime}}^{q^{\prime}}(Z)} \asymp \sup \left\{\left|\sum_k a_k b_k\left(1-\left|z_k\right|\right)\right|:\left\|\left\{a_k\right\}\right\|_{T_p^q(Z)}=1\right\}.
	$$
\end{lemma}

\begin{lemma}[\cite{Aguilar2023}]\label{duality norm}
	Let $Z=\left\{z_k\right\}$ be an $r$-lattice. If $0<p<+\infty$, $0<q<1$, then
	$$
	\sup \left\{\left|\sum_k a_k b_k\left(1-\left|z_k\right|\right)\right|:\left\|\left\{b_k\right\}\right\|_{T_p^q(Z)}=1\right\} \asymp \sup _k\left|a_k\right|\left(1-\left|z_k\right|\right)^{1-1 / q}
	$$
	for any sequence $\left\{a_k\right\}$.
\end{lemma}

\subsection{$p$-Banach space estimates}

We use the \(p\)-Banach-space framework from \cite{IMRN}; only the ingredients needed later are recorded here.

\medskip
\noindent\textbf{Fact 1} (\cite{lvanov2017}).
Let \(0<p,q<\infty\) and \(\alpha>-2\). Then \(\mathcal{A}\mathcal{T}_{q,\alpha}^{p}\) is an \(s\)-Banach space, where \(s=\min\{p,q,1\}\).

\medskip
\noindent\textbf{Fact 2} (\cite{perala2018}).
If \(f\in \mathcal{A}\mathcal{T}_{q,\alpha}^{p}\), then \(f_r(z):=f(rz)\to f\) in \(\mathcal{A}\mathcal{T}_{q,\alpha}^{p}\) as \(r\to1^{-}\).

\medskip
\noindent\textbf{Fact 3} (\cite{IMRN}).
Let \(\{X_n\}_{n\ge0}\) be independent symmetric random variables. If
\[
\sum_{n=0}^\infty a_nX_nz^n\in \mathcal{A}\mathcal{T}_{q,\alpha}^{p}
\quad\text{a.s.},
\]
then the partial sums \(S_n=\sum_{k=0}^n a_kX_kz^k\) converge almost surely in \(\mathcal{A}\mathcal{T}_{q,\alpha}^{p}\).

\medno
The next proposition is the key reduction from the almost-sure statement \(\mathcal Rf\in\mathcal{A}\mathcal{T}_{q,\alpha}^{p}\) to quantitative moment estimates.

\begin{proposition}\label{T:Equivalence}
Let \(0<p,q<\infty\), \(\alpha>-2\), \(\{X_n\}_{n\ge0}\) be a standard random sequence, and
\[
f(z)=\sum_{n=0}^\infty a_nz^n\in H(\mathbb D).
\]
Then the following are equivalent:
\begin{itemize}
\item[($\romannumeral1$)] \(\mathcal Rf\in\mathcal{A}\mathcal{T}_{q,\alpha}^{p}\) almost surely;
\item[($\romannumeral2$)] \(\sup_{n\ge0}\bigl\|\sum_{k=0}^n a_kX_kz^k\bigr\|_{\mathcal{A}\mathcal{T}_{q,\alpha}^{p}}<\infty\) almost surely;
\item[($\romannumeral3$)] \(\mathbb E\bigl(\|\mathcal Rf\|_{\mathcal{A}\mathcal{T}_{q,\alpha}^{p}}^{\,t}\bigr)<\infty\) for every \(t>0\);
\item[($\romannumeral4$)] \(\mathbb E\exp\!\bigl(\lambda\|\mathcal Rf\|_{\mathcal{A}\mathcal{T}_{q,\alpha}^{p}}^{\,s}\bigr)<\infty\) for all sufficiently small \(\lambda>0\), where \(s=\min\{p,q,1\}\).
\end{itemize}
Moreover, \(\bigl(\mathbb E\|\mathcal Rf\|_{\mathcal{A}\mathcal{T}_{q,\alpha}^{p}}^{\,t}\bigr)^{1/t}\) are equivalent for all \(t>0\), with constants depending only on \(t\). The same is true for the Bernoulli, Steinhaus, and Gaussian randomizations.
\end{proposition}

\begin{proof}
The implications \((\romannumeral4)\Rightarrow(\romannumeral3)\Rightarrow(\romannumeral1)\) are immediate, while \((\romannumeral1)\Rightarrow(\romannumeral2)\) follows from Fact~3.
The implication \((\romannumeral2)\Rightarrow(\romannumeral4)\), as well as the equivalence of the moments \(\bigl(\mathbb E\|\mathcal Rf\|_{\mathcal{A}\mathcal{T}_{q,\alpha}^{p}}^{\,t}\bigr)^{1/t}\), is exactly the \(p\)-Banach Fernique--Kahane principle from \cite[Lemmas 9 and 11]{IMRN}.
By Lemma \ref{tent-radial-description}, we have
\begin{align}
E(||\mathcal{R}f||_{\mathcal{A}\mathcal{T}_{q,\alpha}^{p}}^{p})
	&\asymp\int_{\Omega}\int_{\mathbb{T}}\left(\int_{0}^{1}
	|\mathcal{R}f(r\xi)|^{q}(1-r)^{\alpha+1}dr
	\right) ^{\frac{p}{q}}|d\xi| d\mathbb{P} \label{expection}\\ \nonumber
	&\asymp\int_{\mathbb{T}}\left(
	\int_{0}^{1}\left(
	\int_{\Omega}|\mathcal{R}f(r\xi)|^{2}d\mathbb{P}
	\right)^{\frac{q}{2}}
	(1-r)^{\alpha+1}dr
	\right) ^{\frac{p}{q}}|d\xi|\\
	&\asymp\left( \int_{0}^{1}(1-r)^{\alpha+1}
	\left( \sum_{n=0}^{\infty}|a_{n}|^{2}r^{2n}
	\right) ^{\frac{q}{2}}dr
	\right) ^{\frac{p}{q}}. \label{2-trans}
\end{align}
Here $\asymp$ follows from the proof of Theorem \ref{symbol theorem}.
The combination of $(\ref{expection})$ and $(\ref{2-trans})$ implies that, under these three randomization methods,
$E(||\mathcal{R}f||_{\mathcal{A}\mathcal{T}_{q,\alpha}^{p}}^{p})$ only differs by the implict constant in the Khintchine-Kahane inequality (Lemma 11 in \cite{IMRN}).
\end{proof}

	\section{Proof of Theorem \ref{main thm}}\label{proof of bound sec}

\begin{proof}[Proof of Theorem \ref{main thm}] By Lemma \ref{invert} and Lemma \ref{equavilence}, it suffices to consider $t\in \mathbb{R}$ and $\mathfrak{I}_t^\mathrm{RL}$.

	\medno ($\romannumeral1$) If $p<v$, then we consider, 
	for $\gamma<\frac{1}{p}+\frac{\alpha+2}{q},$    the function $$f_{\gamma}(z)=\frac{1}{(1-z)^{\gamma}}.$$
	By Lemma \ref{t function}, we have $f_{\gamma}\in \mathcal{A}\mathcal{T}_{q,\alpha}^{p}.$ By Lemma \ref{t func trans}, we have  $$\big|\mathfrak{I}_{t}^{\mathrm{RL}}f_{\gamma}(z)\big|
	\asymp \Big|\frac{1}{(1-z)^{\gamma-t}}\Big|.$$
	So, together with Lemma \ref{t function}, we   have $$
	\gamma-t<\frac{1}{v}+\frac{\beta+2}{u}.$$
	It follows that the condition $$
	t\geq \frac{\alpha+2}{q}-\frac{\beta+2}{u}+\frac{1}{p}-\frac{1}{v}$$
	is necessary.

	\medno We next turn to the proof of the sufficiency.
	By Lemma \ref{equavilence}, Lemma \ref{tent frac trans} and Lemma \ref{Tent discrete}, we have for $f \in \mathcal{A}\mathcal{T}_{q, \alpha}^p$,
	\begin{align}\label{ac}
		\nonumber	&\int_\mathbb{T} \left(\int_{\Gamma(\xi)}\left|\mathfrak{I}_t f(z)\right|^u(1-|z|)^\beta d A(z)\right)^{\frac{v}{u}}|d \xi|\\
		\nonumber		& \qquad \asymp \int_\mathbb{T}\left(\int_{\Gamma(\xi)}|f(z)|^u(1-|z|)^{\beta+t u} d A(z)\right)^{\frac{v}{u}}|d \xi| \\
		& \qquad \asymp \int_\mathbb{T}\left(\sum_{a_k \in \Gamma(\xi )}\left|f\left(a_k\right)\right|^u\left(1-\left|a_k\right|\right)^{\beta+t u+2}\right)^{\frac{v}{u}}|d \xi|. 
	\end{align}
	Assuming that  $b>\max \left\{\frac{1}{u}, \frac{1}{v}\right\}$, we have
	$$
	\begin{aligned}
		(\ref{ac}) & =\int_\mathbb{T}\left(\sum_{a_k \in \Gamma(\xi) }\left(\left|f\left(a_k\right)\right|^{\frac{1}{b}}\left(1-\left|a_k\right|\right)^{\frac{\beta+t u+2}{b u}}\right)^{b u}\right)^{\frac{b v}{b u}}|d \xi| \\
		& \asymp\left\|\left\{\left|f\left(a_k\right)\right|^{\frac{1}{b}}\left(1-\left|a_k\right|\right)^{\frac{\beta+t u+2}{b u}}\right\}\right\|_{T_{b u}^{b v}(Z)}^{b v}.
	\end{aligned}
	$$
	By Lemma \ref{Tent dual p>1}, we have
	$$\hspace{0.2in} 
	\begin{aligned}
		& \left\|\left\{\left|f\left(a_k\right)\right|^{\frac{1}{b}}\left(1-\left|a_k\right|\right)^{\frac{\beta+t u+2}{b u}}\right\}\right\|_{T_{b u}^{b v}(Z)} \\
		& \qquad =\sup _{\substack{\lambda_k \geq 0 \\
				\left\|\left\{\lambda_k\right\}\right\|_{T_{(b u)^{\prime}}^{(bv)^\prime}(Z)\leq 1}} }  \sum_k \lambda_k\left|f\left(a_k\right)\right|^{\frac{1}{b}}\left(1-\left|a_k\right|\right)^{\frac{\alpha+2}{b q}}\left(1-\left|a_k\right|\right)^{\frac{\beta+t u+2}{b u}+1-\frac{\alpha+2}{b q}}.
	\end{aligned}
	$$
	Therefore, by Lemma \ref{Tent discrete} and Lemma \ref{Tent fraction}, we have 
	$$
	\left\{\lambda_k\left|f\left(a_k\right)\right|^{\frac{1}{b}}\left(1-\left|a_k\right|\right)^{\frac{\alpha+2}{b q}}\right\} \in T_{(b u)^{\prime}}^{(b v)^\prime}(Z) \cdot T_{b q}^{b p}(Z)=T_{\frac{q b u}{u+q(bu-1)}}^{\frac{pbv}{v+p(bv-1)}}(Z).
	$$
	Here the reader  should observe that
	$$
	\left\|\left\{\left|f\left(a_k\right)\right|\left(1-\left|a_k\right|\right)^{\frac{\alpha+2}{q}}\right\}\right\|_{T_q^p(Z)}=\left\|\left\{\left|f\left(a_k\right)\right|^{\frac{1}{b}}\left(1-\left|a_k\right|\right)^{\frac{\alpha+2}{b q}}\right\}\right\|_{T_{b q}^{b p}(Z)}^b.
	$$
	Let $$s=\frac{p b v}{v+p(b v-1)}$$ and $$r=\frac{q b u}{u+q(b u-1)}.$$ Note that, for $p<v$, we have $0<s<1$. Then by Lemma \ref{duality norm}, for $$t\geq\ \frac{1}{p}-\frac{1}{v}+\frac{\alpha+2}{q}-\frac{\beta+2}{u},$$
	$$\hspace{0.2in}
	\begin{aligned}
		& \sum_k \lambda_k\left|f\left(a_k\right)\right|^{\frac{1}{b}}\left(1-\left|a_k\right|\right)^{\frac{\beta+tu+2}{b u}+1}\\
		& \hspace{0.4in} \lesssim\sup _k\left(1-\left|a_k\right|\right)^{\frac{1}{b}\left(t-\frac{1}{p}+\frac{1}{v}-\frac{\alpha+2}{q}+\frac{\beta+2}{u}\right)}\left\|\left\{\lambda_k\left|f\left(a_k\right)\right|^{\frac{1}{b}}\left(1-|a_k|\right)^{\frac{\alpha+2}{b q}}\right\}\right\|_{T_r^s(Z)} \\
		& \hspace{0.4in} \lesssim\left\|\left\{\lambda_k\right\}\right\|_{T_{(b u)^{\prime}}^{(b v)^\prime}}\left\|\left\{\left|f\left(a_k\right)\right|^{\frac{1}{b}}\left(1-\left|a_k\right|\right)^{\frac{\alpha+2}{b q}}\right\}\right\|_{T_{b q}^{b p}(Z)} \\
		&\hspace{0.4in}  \lesssim \left\| \left\{|f\left(a_k\right)| \left(1-\left|a_k\right|\right)^{\frac{\alpha+2}{q}}\right\} \right\|_{T_q^p(Z)}^{1/b} \\
		&\hspace{0.4in}  \asymp \|f\|_{\mathcal{A}\mathcal{T}_{q, \alpha}^p}^{1/b}.
	\end{aligned}
	$$
	The proof of the sufficiency part of ($\romannumeral1$) is complete now.
	
	\medno ($\romannumeral2$)
	If $p\geq v,$ then we consider the function, for $\gamma<-\frac{1}{q},$ $$
	f(z)=\sum_{n=1}^{\infty}n^{\gamma}2^{n\frac{\alpha+2}{q}}z^{2^{n}-1}.$$
	By Lemma \ref{tent lacunary}, we have $f\in \mathcal{A}\mathcal{T}_{q,\alpha}^{p}.$ In order to have $\mathfrak{I}_{t}f\in \mathcal{A}\mathcal{T}_{u,\beta}^{v},$ we need
	$$\{2^{-nt}
	\cdot 2^{n\frac{\alpha+2}{q}}n^{\gamma}
	\cdot 2^{-n\frac{\beta+2}{u}}  \}\in\ell^{u}.$$
	It follows that 
	$$
	\begin{cases}
		t>\frac{\alpha+2}{q}-\frac{\beta+2}{u}; \text{or} \\		
		t=\frac{\alpha+2}{q}-\frac{\beta+2}{u}\  \text{and} \ q\leq u
	\end{cases}$$
	is a necessary condition.
	
	\medno		Now, for the sufficiency part, by Lemma \ref{tent frac trans}, we consider the following:
	$$\hspace{0.2in}
	\begin{aligned}
		\qquad 	&  \int_{\Gamma(\xi)}|f(z)|^u(1-|z|)^{\beta+u t} d A(z) \\
		&	\qquad  \qquad \lesssim \sum_{a_k \in \widetilde{\Gamma}(\xi)}\left(1-\left|a_k\right|\right)^{\beta+u t} \int_{D\left(a_k, r\right)}|f(z)|^u d A(z) \\
		&	\qquad \qquad \lesssim \sum_{a_k \in \widetilde{\Gamma}(\xi)}\left(1-|a_k |\right)^{\beta+u t} \int_{D\left(a_k, r\right)}\left(\int_{D(z, r)}|f(w)|^q d A(w)\right)^{\frac{u}{q}}(1-|z|)^{-\frac{2 u}{q}} d A(z) \\
		&	\qquad  \qquad \lesssim \sum_{a_k \in \widetilde{\Gamma}(\xi)}\left(1-\left|a_k\right|\right)^{\beta+u t-\frac{2 u}{q}+2}\left(\int_{D\left(a_k, 2 r\right)}|f(z)|^q d A(z)\right)^{\frac{u}{q}}.
	\end{aligned}
	$$
	Since $q\leq u$, one has
	$$
	\begin{aligned}
		& \sum_{a_k \in \widetilde{\Gamma}(\xi)}\left(1-\left|a_k\right|\right)^{tq+\beta \frac{q}{u}-2+2 \frac{q}{u}} \int_{D\left(a_k, 2 r\right)}|f(z)|^q d A(z) \\
		&	\qquad  \qquad \lesssim \int_{\widetilde{\widetilde{\Gamma}}(\xi)}|f(z)|^q(1-|z|)^\alpha d A(z).
	\end{aligned}
	$$
	It follows that 
	$$
	\begin{aligned}
		&\int_\mathbb{T} \left(\int_{\Gamma(\xi)}|\mathfrak{I}_t f(z)|^u(1-|z|)^{\beta} d A(z)\right)^{\frac{v}{u}}|d \xi| \\
		&	\qquad  \qquad \lesssim \int_\mathbb{T}\left(\int_{\widetilde{\widetilde{\Gamma}}(\xi)}|f(z)|^q(1-|z|)^\alpha d A(z)\right)^{\frac{v}{q}}|d \xi| \\
		& 	\qquad \qquad \lesssim \int_\mathbb{T}\left(\int_{\Gamma(\xi)}|f(z)|^q(1-|z|)^\alpha d A(z)\right)^{\frac{p}{q}}|d \xi|.
	\end{aligned}
	$$	
	This proof is complete now.
\end{proof}

\begin{proof}[Proof of Theorem \ref{mainthm3}]
	Observe that  $\mathfrak{I}_t^{\mathrm{F}}$ can be factored as the composition of the following:
	\begin{align}\label{1}
		\mathfrak{I}_{-s}^{\mathrm{F}}: \mathcal{F}_{q, \alpha}^p \rightarrow \mathcal{A}\mathcal{T}_{q, q(s-\alpha)-2}^p;
	\end{align}
	\begin{align}\label{2}
		\mathfrak{I}_{t}^{\mathrm{F}}: \mathcal{A}\mathcal{T}_{q, q(s-\alpha)-2}^p \rightarrow \mathcal{A}\mathcal{T}_{u, u(s-\beta)-2}^v;
	\end{align}
	and
	\begin{align}\label{3}
		\mathfrak{I}_{s}^{\mathrm{F}}: \mathcal{A}\mathcal{T}_{u, u(s-\beta)-2}^v \rightarrow \mathcal{F}_{u, \beta}^v.
	\end{align}
	Observe  that   (\ref{1}) and (\ref{3})
	are bounded and invertible. So (\ref{Triebel bound formula}) is
	bounded (resp. compact) if and only if
	(\ref{2}) is bounded (resp. compact). By Theorem \ref{main thm} and Lemma \ref{equavilence}, the desired result follows and the proof is complete now.
\end{proof}

\section{Proof of Theorem \ref{mainthm2}}\label{proof of compact sec}
\begin{proof}[Proof of Theorem \ref{mainthm2}.]
	\noindent By Lemma \ref{invert} and Lemma \ref{equavilence}, it suffices to prove this theorem for $t\in \mathbb{R}$. For  ($\romannumeral1$), we first consider the function $$f_k(z)=\frac{\left(1-\left|a_k\right|^2\right)^\theta}{\left(1-z \bar{a_k}\right)^{\theta+\frac{1}{p}+\frac{\alpha+2}{q}}},$$ where
	$\theta$ is large enough and $|a_k| \rightarrow 1$. Note that $\left\|f_k\right\|_{\mathcal{A}\mathcal{T}_{q,\alpha}^p}\asymp 1$. Assume
	$$
	t \geq \frac{1}{p}-\frac{1}{v}+\frac{\alpha+2}{q}-\frac{\beta+2}{u}.
	$$ Then $$
	\begin{aligned}	\left\|\mathfrak{I}_t f_k(z)\right\|_{\mathcal{A}\mathcal{T}_{u, \beta}^v}^v & \asymp \int_{\mathbb{T}}\left(\int_0^1\left|\mathfrak{I}_t f_k(r \xi)\right|^u(1-r)^{\beta+1} d r\right)^{\frac{v}{u}}|d \xi| \\
		& \asymp\left(1-\left|a_k\right|\right)^{v\left(t-\frac{1}{p}+\frac{1}{v}-\frac{\alpha+2}{q}+\frac{\beta+2}{u}\right)}.
	\end{aligned}
	$$
	It follows that $$t>\frac{1}{p}-\frac{1}{v}+\frac{\alpha+2}{q}-\frac{\beta+2}{u}$$ is a necessary condition.
	For the sufficiency part, let $\left\{f_k\right\}$ be bounded in $\mathcal{A}\mathcal{T}_{q, \alpha}^p$ and converge to zero uniformly on compact subsets of the unit disk.
	For arbitrary $\varepsilon>0$, take $0<s<1$ such that $$1-s<\varepsilon.$$ Then there is an integer $k_0>0$ such that for $k \geq k_0$, $$\sup _{|z|\leq s}\left|f_k(z)\right|<\varepsilon.$$ 
	Let
	$$
	c=\frac{1}{p}+\frac{\alpha+2}{q}-\frac{1}{v}-\frac{\beta+2}{u}.
	$$ By  Lemma \ref{tent frac trans},  Lemma \ref{tent-radial-description} and Corollary \ref{tent embedding}, we have 
	\begin{align}
		\nonumber		\qquad &\left(\int_\mathbb{T} \left(\int_{\Gamma(\xi)}\left|\mathfrak{I}_t f_k(z)\right|^u(1-|z|)^\beta d A(z)\right)^{\frac{v}{u}}|d \xi|\right)^\frac{1}{v}\\
			& \qquad \qquad \asymp\left(\int_\mathbb{T}\left(\int_{\Gamma(\xi)}|f_k(z)|^u(1-|z|)^{\beta+t u} d A(z)\right)^{\frac{v}{u}}|d \xi|\right)^\frac{1}{v} \\
			\label{compact formula}		& \qquad \qquad \lesssim \left(\int_\mathbb{T}\left(\int_0^s\left|f_k(r \xi)\right|^u(1-r)^{\beta+t u+1} d r\right)^{\frac{v}{u}}|d \xi|\right)^{\frac{1}{v}} \\
		& \nonumber\qquad \qquad \qquad +\left(\int_\mathbb{T}\left(\int_s^1\left|f_k(r \xi)\right|^u(1-r)^{\beta+t u+1} d r\right)^{\frac{v}{u}}|d \xi|\right)^{\frac{1}{v}} \\
			&  \qquad\qquad \lesssim \varepsilon+\left(\int_\mathbb{T}\left(\int_s^1\left|f_k(r \xi)\right|^u(1-r)^{\beta+c u+1}(1-r)^{(t-c) u} d r\right)^{\frac{v}{u}}|d \xi|\right)^{\frac{1}{v}} \\
				& \qquad \qquad \lesssim \varepsilon+\varepsilon^{(t-c) u}\left\|f_k\right\|_{\mathcal{A}\mathcal{T}_{q, \alpha}^p}.
	\end{align}
	
	\medno ($\romannumeral2$) For the necessity, we consider the function $$f_n(z)=n^{\frac{\alpha+2}{q}} z^n$$ and note that $\left\|f_n\right\|_{\mathcal{A}\mathcal{T}_{q, \alpha}^p} \asymp 1.$ Then $$\left\|\mathfrak{I}_t f_n(z)\right\|_{\mathcal{A}\mathcal{T}_{u, \beta}^v}\asymp n^{-t+\frac{\alpha+2}{q}-\frac{\beta+2}{u}}.$$ It follows that $$t>\frac{\alpha+2}{q}-\frac{\beta+2}{u}$$ is a necessary condition.
	For the sufficiency, it suffices to consider the quantity in \eqref{compact formula}. Choose $e$ such that $$t>e>\frac{\alpha+2}{q}-\frac{\beta+2}{u}.$$ With a similar argument, by Corollary \ref{tent embedding}, one has
	$$
	\begin{aligned}
		(\ref{compact formula})	&\lesssim \varepsilon+\varepsilon^{(t-e) u}\left(\int_\mathbb{T}\left(\int_s^1\left|f_k(r \xi)\right|^u(1-r)^{\beta+e u+1} d r\right)^{\frac{v}{u}}|d \xi|\right)^{\frac{1}{v}}\\
		&\lesssim \varepsilon+\varepsilon^{(t-e) u}\left\| f_k\right\|_{\mathcal{A}\mathcal{T}_{q, \alpha}^p}.
	\end{aligned}
	$$
	The proof of Theorem \ref{mainthm2} is complete now.
\end{proof}

\medno The proof of Theorem \ref{mainthm4} is similar to that of Theorem \ref{mainthm3}, hence skipped.

\section{Proof of Theorem \ref{symbol theorem}}\label{sec:proof-thm-sym}
	
	\begin{proof}[Proof of Theorem \ref{symbol theorem}]
		We first consider the case $q\leq p.$ By Lemma \ref{tent frac trans}, the Khintchine-Kahane inequality  and Minkowski's inequality, one has
		\begin{align*}
			E(||\mathcal{R}f||_{\mathcal{A}\mathcal{T}_{q,\alpha}^{p}}^{p})
			&\asymp\int_{\Omega}\left( \int_{\mathbb{T}}
			\left( \int_{\Gamma(\xi)}
			|\mathcal{R}f(z)|^{q}(1-|z|)^{\alpha}dA(z)
			\right) ^{\frac{p}{q}}|d\xi|
			\right) d\mathbb{P}\\
			&\asymp\int_{\Omega}\left( 
			\int_{\mathbb{T}}\left( 
			\int_{0}^{1}|\mathcal{R}f(r\xi)|^{q}
			(1-|z|)^{\alpha+1}dr
			\right) ^{\frac{p}{q}}
			|d\xi|\right) d\mathbb{P}\\
			&\leq
			\int_{\mathbb{T}}\left( 
			\int_{0}^{1}\left( 
			\int_{\Omega}|\mathcal{R}f(r\xi)|^{p}d\mathbb{P}
			\right) ^{\frac{q}{p}}
			(1-r)^{\alpha+1}dr
			\right) ^{\frac{p}{q}}|d\xi|\\
			&\asymp
				\int_{\mathbb{T}}\left( 
				\int_{0}^{1}\left(\int_{\Omega} 
				|\mathcal{R}f(r\xi)|^{2}d\mathbb{P}
				\right) ^{\frac{q}{2}}
				(1-r)^{\alpha+1}dr
				\right) ^{\frac{p}{q}}|d\xi|
			\\
			&\asymp
			\left( \int_{0}^{1}M_{2}^{q}(r,f)(1-r)^{\alpha+1}dr
			\right) ^{\frac{p}{q}}.
		\end{align*} 
		It follows that $H(2,q,\frac{\alpha+2}{q})\subseteq (\mathcal{A}\mathcal{T}_{q,\alpha}^{p})_{\star}.$

		\medno 
		For the other direction, by   Lemma \ref{tent frac trans} and the Khintchine-Kahane inequality again,
		\begin{align*}
			E(||\mathcal{R}f||_{\mathcal{A}\mathcal{T}_{q,\alpha}^{p}}^{p})
			&\geq\int_{\mathbb{T}}
			\left( \int_{\Omega}
			\left(\int_{0}^{1}
			|\mathcal{R}f(r\xi)|^{q}
			(1-r)^{\alpha+1}dr
			\right) 
			d\mathbb{P}\right)^{\frac{p}{q}}|d\xi|\\
			&\asymp \int_{\mathbb{T}}\left(
			\int_{0}^{1}\left(\int_{\Omega}
			|\mathcal{R}f(r\xi)|^{2}d\mathbb{P}
			\right)^{\frac{q}{2}}
			(1-r)^{\alpha+1}dr
			\right)^{\frac{p}{q}}|d\xi|\\
			&\asymp \left( 
			\int_{0}^{1}M_{2}^{q}(r,f)(1-r)^{\alpha+1}dr
			\right) ^{\frac{p}{q}}.
		\end{align*} 
		It follows that $(\mathcal{A}\mathcal{T}_{q,\alpha}^{p})_{\star}\subseteq H(2,q,\frac{\alpha+2}{q}).$
		The case $q>p$ follows from similar arguments, hence skipped. 
	\end{proof}
	
\section{Proof of Theorem \ref{embedding theorem}}\label{sec:proof-thm-emdeding}
	\noindent 	We shall need two facts on the embedding between mixed norm spaces, respectively. 
	\begin{lemma}[\cite{Arevalo2015}]\label{mix embedding}
		Let $0<\alpha_1, \alpha_2<\infty$ and $0<p_1, p_2, q_1, q_2 \leq \infty$.
		\begin{itemize}
			\item[($\romannumeral1$)]  If $p_1 \geq p_2$, then $H\left(p_1, q_1, \alpha_1\right) \subseteq H\left(p_2, q_2, \alpha_2\right)$ if and only if
			$$
			\left\{\begin{array}{l}
				\alpha_1<\alpha_2 ; \text { or } \\
				\alpha_1=\alpha_2 \text { and } q_1 \leq q_2 .
			\end{array}\right.
			$$
			
			\item[($\romannumeral2$)] If $p_1<p_2$, then $H\left(p_1, q_1, \alpha_1\right) \subseteq H\left(p_2, q_2, \alpha_2\right)$ if and only if
			$$
			\left\{\begin{array}{l}
				\alpha_1+1 / p_1<\alpha_2+1 / p_2 ; \text { or } \\
				\alpha_1+1 / p_1=\alpha_2+1 / p_2 \text { and } q_1 \leq q_2 .
			\end{array}\right.
			$$
		\end{itemize}
	\end{lemma}	
	
	\medno For part ($\romannumeral1$) of Theorem  \ref{embedding theorem}, we treat first  the necessity. If $p<u,$ then we consider the function (\ref{Fbc}) for $$\frac{\alpha+2}{q}+\frac{1}{p} \qquad \text{and} \qquad c>\frac{1}{p}.$$ By Lemma \ref{Fbc tent}, we have $f \in \mathcal{A}\mathcal{T}_{q,\alpha}^p$. Then by Lemma \ref{Fbc mix}, we  have necessarily $$
	\begin{cases}
		\frac{1}{p}-\frac{1}{u}+\frac{\alpha+2}{q}-\beta=0,
		\	p\leq v; \ or\\
		\frac{1}{p}-\frac{1}{u}+\frac{\alpha+2}{q}-\beta<0.
	\end{cases}$$
	
	\medno To proceed with the sufficiency part of the proof, we divide the proof into two cases.
	
	\bignobf{Case 1:} $\frac{1}{p}-\frac{1}{u}+\frac{\alpha+2}{q}-\beta=0$
	and	$p\leq v$.
	
	\medno	Here we will employ an extrapolation trick, in combination with fractional integration. We first assume $q=2k$. Note that for $f\in H(\mathbb{D})$ $(k\geq 1)$,
	$$
	f\in \mathcal{A}\mathcal{T}_{q,\alpha}^{p}\Leftrightarrow
	f^{k}\in \mathcal{A}\mathcal{T}_{2,\alpha}^{\frac{2p}{q}},$$
	and $$
	f\in H(u,v,\beta)\Leftrightarrow
	f^{k}\in H(\frac{2u}{q},\frac{2v}{q},\frac{q\beta}{2}).$$
	Next, we prove $$\mathcal{A}\mathcal{T}_{2,\alpha}^{\frac{2p}{q}}\subseteq H(\frac{2u}{q},\frac{2v}{q},\frac{q\beta}{2}).$$ By Lemma \ref{Fractional Calderon}, $$\mathfrak{I}_{-\frac{\alpha+2}{2}}^{\mathrm{F}}f\in T_{2,\alpha}^{\frac{2p}{q}}
	\Leftrightarrow f\in H^{\frac{2p}{q}}.$$
	By Theorem 6 in \cite{Guo2}, $$\mathfrak{I}_{-\frac{\alpha+2}{2}}^\mathrm{F}:H^{\frac{2p}{q}}\to
	H(\frac{2u}{q},\frac{2v}{q},\frac{q\beta}{2})$$ is bounded. 
	Then we conclude that for $q=2k$, $$\mathcal{A}\mathcal{T}_{q,\alpha}^{p}\subseteq H(u,v,\beta).$$ For $q<2k$, if we choose $\alpha_k$ such that $$\frac{\alpha+2}{q}=\frac{\alpha_k+2}{2k},$$ then by Corollary \ref{tent embedding}, we have $$
\mathcal{A}\mathcal{T}_{q,\alpha}^{p}\subseteq \mathcal{A}\mathcal{T}_{2k,\alpha_k}^{p}. 
	$$
    Since
$$
\mathcal{A}\mathcal{T}_{2k,\alpha_k}^{p}\subseteq H(u,v,\beta),
$$
it follows 
that
$$
\mathcal{A}\mathcal{T}_{q,\alpha}^{p}\subseteq H(u,v,\beta).
$$

	\bignobf{Case 2:} $\frac{1}{p}-\frac{1}{u}+\frac{\alpha+2}{q}-\beta<0.$
	
	\medno	We first choose $q^\prime>q$ and $\alpha^\prime$ such that $q^\prime>p$ and $$\frac{\alpha+2}{q}=\frac{\alpha^{\prime}+2}{q^{\prime}}.$$ By Corollary \ref{tent embedding}, we have
	$$
	T_{q, \alpha}^p \subseteq T_{q^{\prime}, \alpha^\prime}^p .
	$$
	By Minkowski's inequality and Lemma \ref{mix embedding}
	$$
	T_{q^{\prime}, \alpha^\prime}^p \subseteq H\left(p, q^{\prime}, \frac{\alpha^\prime+2}{q^\prime}\right) \subseteq H(u, v, \beta).
	$$
	
	\bigno ($\romannumeral2$)  For $p=u$, we first consider the  function (\ref{lacunary formula}). By Lemma \ref{tent lacunary} and Lemma \ref{mix lacunary}, we necessarily have $$
	\begin{cases}
		\frac{\alpha+2}{q}-\beta<0;~or\\
		\frac{\alpha+2}{q}-\beta=0,\ v\geq q.
	\end{cases}$$
	Next, we consider the function (\ref{Fbc}). By Lemma \ref{Fbc mix} and Lemma \ref{Fbc tent}, we necessarily have 
	$$
	\begin{cases}
		\frac{\alpha+2}{q}-\beta<0;~or\\
		\frac{\alpha+2}{q}-\beta=0,\ v\geq p.
	\end{cases}$$
	
	\medno For the sufficiency part, we first consider
	$$v \geq \max\{p,q\} \qquad \text{and}  \qquad \frac{\alpha+2}{q}-\beta=0.$$ If we choose $\alpha^\prime$ such that $$\frac{\alpha+2}{q}=\frac{\alpha^\prime+2}{v},$$ then by Corollary \ref{tent embedding},  
	$$ \mathcal{A}\mathcal{T}_{q,\alpha}^p \subseteq \mathcal{A}\mathcal{T}_{v,\alpha^{'}}^p.$$
	By  Minkowski's inequality and Lemma \ref{mix embedding}, one has $$\mathcal{A}\mathcal{T}_{v,\alpha^{'}}^p\subseteq H(p,v,\frac{\alpha^{'}+2}{v}).$$ With the help of part ($\romannumeral1$) of Lemma \ref{mix embedding}, the proof of the sufficiency of the  case $$\frac{\alpha+2}{q}-\beta<0$$ is similar to that of case 2 in part ($\romannumeral1$), hence skipped.

	\medno ($\romannumeral3$) If $p>u,$ then testing the function (\ref{lacunary formula}) yields a necessary condition: $$
	\begin{cases}
		\frac{\alpha+2}{q}-\beta<0;~or\\
		\frac{\alpha+2}{q}-\beta=0,\ v\geq q.
	\end{cases}$$
	
	\medno Now we derive another necessary condition which is perhaps the most roundabout part of the proof. Assume $$p>u, \ v\geq q \quad \text{and} \quad \beta=\frac{\alpha+2}{q}, \quad \text{but} \quad  v<u.$$ By  Theorem 6 in \cite{Guo2}, we conclude that 
	$$\mathfrak{I}_{-\frac{\alpha+2}{2}}: H^{\frac{2 p}{q}} \rightarrow H\left(\frac{2u}{q}, \frac{2 v}{q}, \frac{q \beta}{2}\right)$$ is  unbounded.
	By Lemma  \ref{Fractional Calderon}, we have that 
	$$\mathfrak{I}_{\frac{\alpha+2}{2}}^\mathrm{F}: \mathcal{A}\mathcal{T}_{2,\alpha}^{\frac{2p}{q}}  \rightarrow H^{\frac{2 p}{q}}$$
	is bounded and invertible.
	It follows that $$\mathcal{A}\mathcal{T}_{2, \alpha}^{\frac{2 p}{q}} \nsubseteq H\left(\frac{2 u}{q}, \frac{2 v}{q}, \frac{a \beta}{2}\right).$$ Here we first assume $q=\frac{2}{k}$. 
	Then we choose $$f \in \mathcal{A}\mathcal{T}_{2, \alpha}^{\frac{2p}{q}} \qquad \text{but} \qquad f \notin H\left(\frac{2 u}{q}, \frac{2 v}{q}, \frac{q \beta}{2}\right).$$
	Let $$g=f^k.$$ Then $$g\in \mathcal{A}\mathcal{T}_{q, \alpha}^p$$
	but $$g \notin H(u, v, \beta).$$ It follows that
	for $$q=\frac{2}{k}, \qquad k \geq 1, \qquad q \leq v<u<p$$ and  $\frac{\alpha+2}{q}=\beta$, we have
	$$
	\mathcal{A}\mathcal{T}_{q, \alpha}^p \nsubseteq H(u, v, \beta).
	$$
	Next, we consider $q>\frac{2}{k}$. If we choose $$
	\frac{k}{2}\left(\alpha_k+2\right)=\frac{\alpha+2}{q},
	$$ then by Corollary \ref{tent embedding}, 	
	$$
	\mathcal{A}\mathcal{T}_{\frac{2}{k}, \alpha_k}^p \subseteq \mathcal{A}\mathcal{T}_{q, \alpha}^p .
	$$
	But $$\mathcal{A}\mathcal{T}_{\frac{2}{k},\alpha_k }^p \nsubseteq H(u, v, \beta),$$ since $$p>u,  \qquad \frac{2}{k}<q \leq v \qquad \text{and} \qquad u>v.$$
	So $$T_{q, \alpha}^p  \nsubseteq H(u, v, \beta).$$
	In summary, we have derived the following necessary condition: 
	$$
	\begin{cases}
		\frac{\alpha+2}{q}-\beta<0;~\text{or}\\
		\frac{\alpha+2}{q}-\beta=0,\ v\geq \max\{q,u\}.
	\end{cases}$$
	
	\medno 
	Next, for the sufficiency part, we first consider $$v\geq \max \{q,u\} \qquad \text{and} \qquad \frac{\alpha+2}{q}-\beta=0.$$ If we choose $\alpha^\prime$ such that $$\frac{\alpha+2}{q}=\frac{\alpha^{\prime}+2}{v},$$ then   by Corollary \ref{tent embedding} and Minkowski's inequality, 
	$$
	\mathcal{A}\mathcal{T}_{q, \alpha}^p \subseteq \mathcal{A}\mathcal{T}_{v, \alpha^{\prime}}^p \subseteq \mathcal{A}\mathcal{T}_{v, \alpha^{\prime}}^u \subseteq H(u, v, \beta).
	$$
	With the help of part ($\romannumeral2$) of Lemma \ref{mix embedding}, the proof of the sufficiency of  the case $$\frac{\alpha+2}{q}-\beta<0$$ is similar to that of case 2 in part ($\romannumeral1$), hence skipped.
	This proof of Theorem \ref{embedding theorem} is complete now.
	
	\bignobf{Remark.} If one follows the \emph{``usual''} approach to embedding problems, then he/she can derive the necessary condition for $p>u$ in one shot if  a function $g \in H(\mathbb{D})$  such that $$ g \in \mathcal{A}\mathcal{T}_{q, \alpha}^p,\qquad \text{but} \qquad g \notin H(u, v, \beta),$$ where $$0<q \leq v<u<p<\infty \qquad \text{and} \qquad \frac{\alpha+2}{q}=\beta,$$ can be found. Such an example is still desirable to us but remains elusive to our repeated efforts. So, we offer

	\bignobf{Problem B.}  Find a nonzero function $f \in H(\mathbb{D})$ such that $$f \in \mathcal{A}\mathcal{T}_{2, \alpha}^s, \qquad \text{but} \qquad f \notin H\left(d, e, \frac{\alpha+2}{2}\right),$$ where $s>d>e \geq 2$ and $\alpha>-2$.

\providecommand{\bysame}{\leavevmode\hbox to3em{\hrulefill}\thinspace}
\providecommand{\MR}{\relax\ifhmode\unskip\space\fi MR }
\providecommand{\MRhref}[2]{%
  \href{http://www.ams.org/mathscinet-getitem?mr=#1}{#2}
}
\providecommand{\href}[2]{#2}


\begin{thebibliography}{10}

\bibitem{Aguilar2023}
T.~Aguilar-Hern\'{a}ndez and P.~Galanopoulos, \emph{Average radial
  integrability spaces, tent spaces and integration operators}, J. Math. Anal.
  Appl. \textbf{523} (2023), no.~2, Paper No. 127028, 38 pp.

\bibitem{Aguilar20231}
\bysame, \emph{Inequalities on tent spaces and closed range integration
  operators on spaces of average radial integrability}, Rev. Real Acad. Cienc.
  Exactas F{\'\i}s. Nat. Ser. A Mat. \textbf{119} (2025), Paper No. 70, 26 pp.

\bibitem{Arevalo2015}
I.~Ar\'{e}valo, \emph{A characterization of the inclusions between mixed norm
  spaces}, J. Math. Anal. Appl. \textbf{429} (2015), no.~2, 942--955.

\bibitem{Ars1999}
M.~Arsenovi\'{c}, \emph{Embedding derivatives of {${M}$}-harmonic functions
  into {$L^p$}-spaces}, Rocky Mountain J. Math. \textbf{29} (1999), no.~1,
  61--76.

\bibitem{Avetisyan2012}
K.L. Avetisyan, \emph{A note on mixed norm spaces of analytic functions}, Aust.
  J. Math. Anal. Appl. \textbf{9} (2012), no.~1, Paper No. 16, 6 pp.

\bibitem{Buck}
S.M. Buckley, P.~Koskela, and D.~Vukoti\'{c}, \emph{Fractional integration,
  differentiation, and weighted {B}ergman spaces}, Math. Proc. Cambridge
  Philos. Soc. \textbf{126} (1999), no.~2, 369--385.

\bibitem{Cinlar}
E.~Cinlar, \emph{Probability and stochastics}, Graduate Texts in Mathematics,
  vol. 261, Springer, New York, 2011.

\bibitem{Cohn2000}
W.S. Cohn and I.E. Verbitsky, \emph{Factorization of tent spaces and {H}ankel
  operators}, J. Funct. Anal. \textbf{175} (2000), no.~2, 308--329.

\bibitem{Coi1985}
R.R. Coifman, Y.~Meyer, and E.M. Stein, \emph{Some new function spaces and
  their applications to harmonic analysis}, J. Funct. Anal. \textbf{62} (1985),
  no.~2, 304--335.

\bibitem{IMRN}
X.~Fang, G.~Cheng, and C.~Liu, \emph{A {L}ittlewood-type theorem for random
  {B}ergman functions}, Int. Math. Res. Not. (2022), no.~14, 11056--11091.

\bibitem{Guo1}
X.~Fang, F.~Guo, S.~Hou, and X.~Zhu, \emph{Fractional integration on mixed norm
  spaces. {I}}, Complex Anal. Oper. Theory \textbf{18} (2024), no.~3, Paper No.
  45, 21 pp.

\bibitem{Guo3}
\bysame, \emph{Fractional {V}olterra-type operators on {H}ardy spaces}, Studia
  Math. \textbf{283} (2025), no.~1, 1--42.

\bibitem{Flett2}
T.M. Flett, \emph{Mean values of power series}, Pacific J. Math. \textbf{25}
  (1968), 463--494.

\bibitem{Flett1971}
\bysame, \emph{Temperatures, {B}essel potentials and {L}ipschitz spaces}, Proc.
  London Math. Soc. \textbf{22} (1971), 385--451.

\bibitem{Flett1}
\bysame, \emph{The dual of an inequality of {H}ardy and {L}ittlewood and some
  related inequalities}, J. Math. Anal. Appl. \textbf{38} (1972), 746--765.

\bibitem{jh1892}
J.~Hadamard, \emph{Essai sur l'{\'e}tude des fonctions donn{\'e}es par leur
  d{\'e}veloppement de taylor}, J. de Math. \textbf{8} (1892), 101--186.

\bibitem{HL}
G.H. Hardy and J.E. Littlewood, \emph{Some properties of fractional integrals.
  {II}}, Math. Z. \textbf{34} (1932), no.~1, 403--439.

\bibitem{lvanov2017}
K.~Ivanov and P.~Petrushev, \emph{Harmonic {B}esov and {T}riebel-{L}izorkin
  spaces on the ball}, J. Fourier Anal. Appl. \textbf{23} (2017), no.~5,
  1062--1096.

\bibitem{Jev1996}
M.~Jevti\'{c}, \emph{Embedding derivatives of {$ M$}-harmonic {H}ardy spaces
  {$H^p$} into {L}ebesgue spaces, {$0<p<2$}}, Rocky Mountain J. Math.
  \textbf{26} (1996), no.~1, 175--187.

\bibitem{Kahane1985}
J.P. Kahane, \emph{Some random series of functions}, second ed., Cambridge
  Studies in Advanced Mathematics, vol.~5, Cambridge University Press,
  Cambridge, 1985.

\bibitem{Kim}
H.~Kim, \emph{Derivatives of {B}laschke products}, Pacific J. Math.
  \textbf{114} (1984), no.~1, 175--190.

\bibitem{Littlewood1926}
J.~E. Littlewood, \emph{On the mean values of power series}, J. London Math.
  Soc. \textbf{25} (1926), 328--337.

\bibitem{Littlewood1930}
\bysame, \emph{On mean values of power series {II}}, J. London Math. Soc.
  \textbf{5} (1930), 179--182.

\bibitem{Lueck1991}
D.H. Luecking, \emph{Embedding derivatives of {H}ardy spaces into {L}ebesgue
  spaces}, Proc. London Math. Soc. \textbf{63} (1991), no.~3, 595--619.

\bibitem{MM1}
M.~Mateljevi\'{c} and M.~Pavlovi\'{c}, \emph{{$L^{p}$}-behavior of power series
  with positive coefficients and {H}ardy spaces}, Proc. Amer. Math. Soc.
  \textbf{87} (1983), no.~2, 309--316.

\bibitem{Wang2020}
S.~Miihkinen, J.~Pau, A.~Per\"{a}l\"{a}, and M.~Wang, \emph{Volterra type
  integration operators from {B}ergman spaces to {H}ardy spaces}, J. Funct.
  Anal. \textbf{279} (2020), no.~4, Paper No. 108564, 32 pp.

\bibitem{Oswald1983}
P.~Oswald, \emph{On {B}esov-{H}ardy-{S}obolev spaces of analytic functions in
  the unit disc}, Czechoslovak Math. J. \textbf{33} (1983), no.~3, 408--426.

\bibitem{Paley19302}
R.E.A.C. Paley and A.~Zygmund, \emph{On some series of functions, (2)}, Proc.
  Camb. Phil. Soc. \textbf{26} (1930), 458--474.

\bibitem{Pavlovic2013}
M.~Pavlovi\'{c}, \emph{On the {L}ittlewood-{P}aley {$g$}-function and
  {C}alder\'{o}n's area theorem}, Expo. Math. \textbf{31} (2013), no.~2,
  169--195.

\bibitem{pm2019}
\bysame, \emph{Function classes on the unit disc. an introduction}, 2nd ed., De
  Gruyter Studies in Mathematics, vol.~52, De Gruyter, Berlin, 2019.

\bibitem{Pelaez20152}
J.~Pel{\'a}ez and J.~R{\"a}tty{\"a}, \emph{Embedding theorems for {B}ergman
  spaces via harmonic analysis}, Math. Ann. \textbf{362} (2015), 205--239.

\bibitem{perala2018}
A.~Perälä, \emph{Duality of holomorphic {H}ardy type tent spaces},
  arXiv:1803.10584v1 (2018).

\bibitem{zhu2005}
K.~Zhu, \emph{Spaces of holomorphic functions in the unit ball}, Graduate Texts
  in Mathematics, vol. 226, Springer-Verlag, New York, 2005.

\bibitem{Guo2}
X.~Zhu, X.~Fang, F.~Guo, and S.~Hou, \emph{Fractional integration on mixed norm
  spaces. {II}}, J. Geom. Anal. \textbf{33} (2023), no.~5, Paper No. 158, 37
  pp.

\end{thebibliography}
\end{document}